\documentclass[12pt]{article}
\usepackage{amsthm,amssymb, amsmath, amscd, amsfonts,graphicx, latexsym,verbatim}
\usepackage{cite, psfrag,color}
\usepackage{titlesec}
\usepackage{epstopdf}
\usepackage{fancybox}
\usepackage{float}
\usepackage[numbers,square]{natbib}
\usepackage[font=normalsize,labelsep=none]{caption}
\usepackage{diagbox}
\usepackage{changepage}
\usepackage{enumerate}
\usepackage{arydshln}
\usepackage{bbm}
\usepackage{setspace}
\usepackage{stmaryrd}  
\usepackage{color}
\usepackage{enumitem}
\usepackage{diagbox}
\usepackage{changepage}
\usepackage{enumerate}
\usepackage{arydshln}
\usepackage{bbm}
\usepackage{setspace}
\usepackage{titletoc}
\usepackage{enumitem}
\usepackage{cite} 
\usepackage{mathtools} 
\usepackage{booktabs}
\usepackage[
pagebackref=true,
colorlinks=true,
linkcolor=red,
citecolor=red
]{hyperref}

\titleformat{\section} 
{\raggedright\large\bfseries\bf} 
{\thesection .\quad} 
{0pt} 
{}  

\titleformat{\subsection} 
{\raggedright\normalsize\bfseries\it} 
{\thesubsection .\quad} 
{0pt} 
{} 

 \newtheorem{theorem}{Theorem}[section]
 \newtheorem{lemma}{Lemma}[section]
 \newtheorem{corollary}{Corollary}[section]

 \newtheorem{proposition}{Proposition}[section]
 
\theoremstyle{definition} 
\newtheorem{remark}{Remark}[section]

\theoremstyle{definition} 
\newtheorem{definition}{Definition}[section]


\newlength{\boxedparwidth}
  {\begin{center} \begin{tabular}{|@{\hspace{.315in}}c@{\hspace{.15in}}|}
                  \hline \\ \begin{minipage}[t]{\boxedparwidth}
                  \setlength{\parindent}{.25in}}%
  {\end{minipage} \\ \\ \hline \end{tabular} \end{center}}
\newtheoremstyle{boldname}
{10pt}   
{10pt}   
{\normalfont} 
{}      
{\normalfont} 
{}     
{ }     
{\textbf{#1}~\textbf{{#2}}\textbf{.}\thmnote{ (#3)}} 

\theoremstyle{boldname}
\newtheorem{example}{Example}[section]

\newcommand{\smallsubseteq}{\mathrel{\scalebox{0.55}{$\subseteq$}}}
\newcommand{\smalleq}{\mathrel{\scalebox{0.55}{$=$}}}

\newcommand{\tinysubseteq}{\mathrel{\scalebox{0.45}{$\subseteq$}}}
\newcommand{\tinyeq}{\mathrel{\scalebox{0.45}{$=$}}}

\usepackage{relsize}
\DeclareMathOperator*{\Uplusmed}{\mathlarger{\mathlarger{\uplus}}}

\begin{document}
\begin{center}
{\large{\bf{
Mahonian statistics on words with fixed weak right-to-left minima and on permutations with a fixed descent set
\vskip 0mm}}}
\end{center}
\vspace{-2.0mm}
\begin{center}
\text{Shao-Hua Liu}\\
   \vskip 1.0mm
{School of Statistics and Data Science\\
Guangdong University of Finance and Economics\\
Guangzhou, China\\

   \vskip 1.0mm
Email: liushaohua@gdufe.edu.cn}
\end{center}
 \title{}

 \vskip 1mm

\noindent {\bf Abstract.}
Our first main result shows that, for words with a fixed multiset of weak right-to-left minima, 
the statistics within each of the following three classes are equidistributed:

1. Mahonian statistics:
\textsf{inv}, \textsf{maj}, \textsf{den}, \textsf{mak}, \textsf{mad}, $\textsf{inv}_{r}$, $r\textsf{maj}$,  and $r$\textsf{den};

2. Euler--Mahonian statistics:  $(\textsf{des},\textsf{maj})$,
$(\textsf{exc},\textsf{den})$,
and $(\textsf{des},\textsf{mak})$;

3. $r$-Euler--Mahonian statistics: 
$(r\textsf{des},r\textsf{maj})$ and
$(r\textsf{exc},r\textsf{den})$.

\noindent Our second main result shows that, 
for permutations with a fixed descent set,
the statistics within each of the following two classes are equidistributed:

1. Mahonian statistics: \textsf{inv}, \textsf{imaj}, $\textsf{imak}$, $\textsf{iinv}_{r}$, and $\textsf{istat}$;

2. Mahonian--Stirling-type statistics: 
$(\textsf{inv},\textsf{lrmax})$,
$(\textsf{imaj},\textsf{lrmax})$,
$(\textsf{imak},\textsf{lrmax})$,
and $(\textsf{iinv}_{r},\textsf{lrmax})$.

\noindent 
Moreover, we apply these results to set partitions, 221-avoiding words,  alternating permutations, and permutations with $k$ alternating runs, 
thereby obtaining several families of equidistributed Mahonian-type statistics on these combinatorial structures.


   \vskip 0mm
   
\titlecontents{section}[1.5em]
{\footnotesize \vspace{-6.0pt}}
{\contentslabel{1.5em}}{\hspace*{-1.6em}}
{~\titlerule*[0.6pc]{$.$}~\contentspage}

\titlecontents{subsection}[3.5em]
{\footnotesize\vspace{-6.0pt}}
{\contentslabel{2.3em}}{\hspace*{-4em}}
{~\titlerule*[0.6pc]{$.$}~\contentspage}

\tableofcontents    
   
\section{Introduction}
Throughout the paper, we let $n$ be a positive integer,
and let $[n]=\{1,2,\ldots,n\}$.
A \emph{composition} of $n$ is a finite sequence $\alpha=(\alpha_{1},\alpha_{2},\ldots,\alpha_{m})$ of positive integers satisfying $\alpha_{1}+\alpha_{2}+\cdots+\alpha_{m}=n$.
We write $\alpha\models n$ to indicate that $\alpha$ is a composition of $n$.
Unless otherwise noted, we shall always assume that $\alpha=(\alpha_{1},\alpha_{2},\ldots,\alpha_{m})\models n$.

Let $\mathfrak{S}_{\alpha}$ denote the set of multipermutations 
(or words) of the multiset
$\{1^{\alpha_{1}},2^{\alpha_{2}},\ldots,m^{\alpha_{m}}\}$,
where $i^{\alpha_{i}}$ denotes $\alpha_{i}$ copies of $i$.
For example, $\mathfrak{S}_{(3,1)}=\{1112,1121,1211,2111\}$.
It is well known that the cardinality of $\mathfrak{S}_{\alpha}$, denoted by $|\mathfrak{S}_{\alpha}|$, is given by the \emph{multinomial coefficient}
$$|\mathfrak{S}_{\alpha}|
=\binom{n}{\alpha_{1},\alpha_{2},\dots,\alpha_{m}}=\frac{n!}{\alpha_{1}!\alpha_{2}!\ldots \alpha_{m}!}.$$	
   \vskip -3mm

Given a word 
$w=w_{1}w_{2}\ldots w_{n}\in\mathfrak{S}_{\alpha}$,
a position $i$ with $1\leqslant i\leqslant n-1$ is called a \emph{descent} of $w$ if $w_{i}>w_{i+1}$.
Let $\textsf{Des}(w)$ denote the set of descents of $w$,
and let $\textsf{des}(w)=|\textsf{Des}(w)|$.
Let $\overline{w}=a_{1}a_{2}\ldots a_{n}$ denote the nondecreasing rearrangement of $w=w_{1}w_{2}\ldots w_{n}$. 
An \emph{excedance} of $w$ is a position $i$ such that $w_{i}>a_{i}$.
The number of excedances of $w$ is denoted by $\textsf{exc}(w)$.
A classical result states that $\textsf{des}$ and $\textsf{exc}$ are equidistributed on $\mathfrak{S}_{\alpha}$; that is,
$$\sum_{w\in\mathfrak{S}_{\alpha}}q^{\textsf{des}(w)}=\sum_{w\in\mathfrak{S}_{\alpha}}q^{\textsf{exc}(w)}.$$
Any statistic that is equidistributed with them is said to be \emph{Eulerian}.

An \emph{inversion} of
$w=w_{1}w_{2}\ldots w_{n}\in\mathfrak{S}_{\alpha}$ is a pair $(i, j)$ with $i<j$ and $w_{i}>w_{j}$. 
Let $\textsf{inv}(w)$ denote the number of inversions of $w$.
The \emph{major index} of $w$, denoted by $\textsf{maj}(w)$,
is the sum of the descents of $w$; that is,
$$\textsf{maj}(w)=\sum_{i\in {\textsf{Des}}(w)}i.$$
MacMahon \cite{MacMahon-1916-book} showed that the statistics \textsf{inv} and \textsf{maj} are equidistributed on $\mathfrak{S}_{\alpha}$,
with common generating function given by the \emph{$q$-multinomial coefficient}: 
\begin{align}\label{eq-inv-maj-S_M}
	\sum_{w\in\mathfrak{S}_{\alpha}}q^{\textsf{inv}(w)}=\sum_{w\in\mathfrak{S}_{\alpha}}q^{\textsf{maj}(w)}
	=\genfrac{[}{]}{0pt}{}{n}{\alpha_{1},\alpha_{2},\dots,\alpha_{m}}_{q}
	=\frac{[n]!_{q}}{[\alpha_{1}]!_{q}[\alpha_{2}]!_{q}\ldots [\alpha_{m}]!_{q}},
\end{align}
where $[i]!_{q}=[i]_{q}[i-1]_{q}\ldots[1]_{q}$,
and 
$[i]_{q}=1+q+\cdots+q^{i-1}$.
In his honor, any statistic with this distribution is said to be \emph{Mahonian}.
Some common Mahonian statistics in the literature are listed in Table \ref{Table-1}.
Since we do not use $\textsf{\footnotesize Z}$ and $\textsf{sor}$, 
we omit their definitions; 
the definition of \textsf{stat} for words can be found in \cite{Kitaev-2016}. 
Definitions of the remaining statistics in Table \ref{Table-1} are given in the corresponding proof sections.

\begin{table}[t]
	\centering
	\begin{tabular}{llll}
		Name& Reference & Setting (permutations/words) & Year \\
		\hline
		\textsf{inv}&Rodrigues \cite{Rodrigues-1839}& permutations &1839\\
		~           &MacMahon \cite{MacMahon-1916-book}& words &1916\\
		\textsf{maj}&MacMahon \cite{MacMahon-1916-book}&permutations \& words &1916\\
		$r\textsf{maj}$ &Rawlings \cite{Rawlings-1981} &permutations& 1981\\
		~ & Rawlings \cite{Rawlings-1981-2} & words& 1981  \\
		\textsf{inv}$_{r}$&Kadell \cite{Kadell-1985}&permutations \& words &1985\\
		\textsf{\footnotesize Z}  &Zeilberger--Bressoud \cite{Zeilberger-1985}&words &1985  \\
		\textsf{den} &Denert \cite {Denert-1990}, Foata--Zeilberger \cite{Foata-Zeilberger-1990} & permutations& 1990\\
		~ & Han \cite{Han-1994} &words& 1994  \\
		\textsf{mak}  & Foata--Zeilberger \cite{Foata-Zeilberger-1990} &  permutations& 1990 \\
		~ &Clarke--Steingr\'{\i}msson--Zeng \cite{Clarke-1997} & words & 1997 \\
		\textsf{mad}  &Clarke--Steingr\'{\i}msson--Zeng \cite{Clarke-1997}& permutations \& words &1997  \\
		\textsf{stat}  & Babson--Steingr\'{\i}msson \cite{Babson-2000} & permutations&  2000 \\
		~ & Kitaev--Vajnovszki \cite{Kitaev-2016} & words &  2016 \\
		\textsf{sor}  &  Petersen \cite{Petersen-2011} & permutations&  2011 \\
~ &  Grady--Poznanovi\'{c} \cite{Grady-2018} & words &  2018 \\
$r$\textsf{den}  &  Han \cite{Han-1991} & permutations &  1991 \\
~ &  Huang--Lin--Yan \cite{Huang-Lin-Yan-2026}  & words &  2026 		
	\end{tabular}
	\caption{Some Mahonian statistics.}\label{Table-1}
\end{table}

A bistatistic is called \emph{Euler--Mahonian} if it is equidistributed with $(\textsf{des}, \textsf{maj})$. 
Foata and Zeilberger \cite{Foata-Zeilberger-1990} showed that $(\textsf{exc}, \textsf{den})$ is Euler--Mahonian on permutations, and Han \cite{Han-1994} extended this result to words. 
Foata and Zeilberger \cite{Foata-Zeilberger-1990} also introduced another Euler--Mahonian statistic, namely $(\textsf{des}, \textsf{mak})$, on permutations, 
which was extended to words by Clarke, Steingr\'{\i}msson, and Zeng \cite{Clarke-1997}.
The above results can be expressed in terms of generating functions as follows:
\begin{align}\label{eq-Euler-Mahonian}
	\sum_{w\in\mathfrak{S}_{\alpha}}t^{\textsf{des}(w)}q^{\textsf{maj}(w)}=
	\sum_{w\in\mathfrak{S}_{\alpha}}t^{\textsf{exc}(w)}q^{\textsf{den}(w)}=
	\sum_{w\in\mathfrak{S}_{\alpha}}t^{\textsf{des}(w)}q^{\textsf{mak}(w)}.
\end{align}

Let $r$ be a positive integer.
A bistatistic is called \emph{$r$-Euler--Mahonian} if it is equidistributed with $(r\textsf{des}, r\textsf{maj})$. 
Liu \cite{Liu-2024} proved that $(r\textsf{exc}, r\textsf{den})$ is $r$-Euler--Mahonian on permutations, and 
Huang, Lin, and Yan \cite{Huang-Lin-Yan-2026} extended this result to words:
\begin{align}\label{r-eq-Euler-Mahonian}
	\sum_{w\in\mathfrak{S}_{\alpha}}t^{r\textsf{des}(w)}q^{r\textsf{maj}(w)}=
	\sum_{w\in\mathfrak{S}_{\alpha}}t^{r\textsf{exc}(w)}q^{r\textsf{den}(w)}.
\end{align}

The  remainder of this paper is organized as follows.
In Section \ref{Section-restricted-words},
we present our main results on words with a fixed multiset of weak right-to-left minima; 
their proofs are deferred to Sections \ref{section-inv-majd}--\ref{Section-tri}.
In Section \ref{Section-restricted-permutations},
we establish our main results on permutations with a fixed descent set.
Sections \ref{Section-set partitions}--\ref{Section-Permutations-with-k-runs} are devoted, respectively, 
to applications to set partitions, $221$-avoiding words, alternating permutations, and permutations with $k$ alternating runs.

\section{Mahonian statistics on words with a fixed multiset of weak right-to-left minima}\label{Section-restricted-words}
\subsection{The main results}
Let $w=w_{1}w_{2}\ldots w_{n}\in\mathfrak{S}_{\alpha}$.
A \emph{(strict) right-to-left minimum} of $w$ is a letter $w_{i}$ such that either $i=n$ or $w_{i}<w_{j}$ for all $j>i$.
Let $\textsf{Rlmin}(w)$ denote the set of right-to-left minima of $w$,
and let $\textsf{rlmin}(w)=|\textsf{Rlmin}(w)|$.
A \emph{weak right-to-left minimum} of $w$ is a letter $w_{i}$ such that either $i=n$ or $w_{i}\leqslant w_{j}$ for all $j>i$.
Let $\textsf{Rlwmin}(w)$  denote the multiset of weak right-to-left   minima of $w$,
and let $\textsf{rlwmin}(w)=|\textsf{Rlwmin}(w)|$.
For example, if $\sigma=3212315354646547577$, then we have
\begin{align*}
\textsf{Rlmin}(32123\underline{1}5\underline{3}546465\underline{4}7\underline{5}7\underline{7})&=\{1,3,4,5,7\}~\text{and}~\textsf{rlmin}(\sigma)=5;\\
\textsf{Rlwmin}(32\underline{1}23\underline{1}5\underline{3}5\underline{4}6\underline{4}65\underline{4}7\underline{577})&=\{1,1,3,4,4,4,5,7,7\}~\text{and}~\textsf{rlwmin}(\sigma)=9.
\end{align*}
Let $\text{supp}(M)$ denote the set of elements of the multiset $M$, ignoring multiplicities.
For example, $\text{supp}(\{1, 1, 1, 3, 3, 4, 4\}) = \{1, 3, 4\}$.
Clearly, we have $\text{supp}(\textsf{Rlwmin}(w))=\textsf{Rlmin}(w).$

Let $\mathfrak{S}_{\alpha,R}$ denote the set of words $w\in\mathfrak{S}_{\alpha}$ such that $\textsf{Rlwmin}(w)=R$.
We show that the Mahonian statistics listed in Table \ref{Table-1} are equidistributed on $\mathfrak{S}_{\alpha,R}$,
except for \textsf{\footnotesize Z}, \textsf{stat}, and \textsf{sor}.
\begin{theorem}\label{Thm-main-Mahonian-Rlwmin}
For any multiset $R$ and positive integer $r$,
we have
	\begin{align*} 
		\begin{split}
			&\sum_{w\in\mathfrak{S}_{\alpha,R}}q^{\emph{\textsf{inv}}(w)~}
			=\sum_{w\in\mathfrak{S}_{\alpha,R}}q^{\emph{\textsf{maj}}(w)~}
			=\sum_{w\in\mathfrak{S}_{\alpha,R}}q^{\emph{\textsf{den}}(w)~}
			=\sum_{w\in\mathfrak{S}_{\alpha,R}}q^{\emph{\textsf{mak}}(w)}
			=\sum_{w\in\mathfrak{S}_{\alpha,R}}q^{\emph{\textsf{mad}}(w)}\\
			=&\sum_{w\in\mathfrak{S}_{\alpha,R}}q^{\emph{\textsf{inv}}_{r}(w)}
            =\sum_{w\in\mathfrak{S}_{\alpha,R}}q^{r\emph{\textsf{maj}}(w)}	   
			=\sum_{w\in\mathfrak{S}_{\alpha,R}}q^{r\emph{\textsf{den}}(w)}.     		
		\end{split}
	\end{align*}
\end{theorem}
\begin{remark}
The equidistribution of \textsf{inv} and \textsf{maj} on $\mathfrak{S}_{\alpha,R}$ already appears in ${\text{Foata~and~Han \cite{Foata-2011}}}$.
\end{remark}
\begin{remark}
None of the statistics \textsf{\footnotesize Z}, \textsf{stat}, or \textsf{sor} is equidistributed with \textsf{inv} on $\mathfrak{S}_{\alpha,R}$.
\end{remark}
For Euler--Mahonian statistics, 
we establish the following refinement of (\ref{eq-Euler-Mahonian}).
\begin{theorem}\label{Thm-main-Euler--Mahonian-Rlwmin}
For any multiset $R$, we have
	\begin{align*}
		\sum_{w\in\mathfrak{S}_{\alpha,R}}t^{\emph{\textsf{des}}(w)}q^{\emph{\textsf{maj}}(w)}=
		\sum_{w\in\mathfrak{S}_{\alpha,R}}t^{\emph{\textsf{exc}}(w)}q^{\emph{\textsf{den}}(w)}=
		\sum_{w\in\mathfrak{S}_{\alpha,R}}t^{\emph{\textsf{des}}(w)}q^{\emph{\textsf{mak}}(w)}.
	\end{align*}
\end{theorem}
For $r$-Euler--Mahonian statistics,
we establish the following refinement of  (\ref{r-eq-Euler-Mahonian}).
\begin{theorem}\label{Thm-main-r-Euler--Mahonian-Rlwmin}
For any multiset $R$ and positive integer $r$, we have
	\begin{align*}
		\sum_{w\in\mathfrak{S}_{\alpha,R}}t^{r\emph{\textsf{des}}(w)}q^{r\emph{\textsf{maj}}(w)}=
		\sum_{w\in\mathfrak{S}_{\alpha,R}}t^{r\emph{\textsf{exc}}(w)}q^{r\emph{\textsf{den}}(w)}.
	\end{align*}
\end{theorem}

The equidistribution results needed for Theorems \ref{Thm-main-Mahonian-Rlwmin}--\ref{Thm-main-r-Euler--Mahonian-Rlwmin} are established in Sections \ref{section-inv-majd}--\ref{Section-tri}.

\subsection{Several immediate corollaries}
Throughout this paper, we use $\Uplusmed$ to denote disjoint union.
Clearly, 	
\begin{align*} 
	\mathfrak{S}_{\alpha}=
	\Uplusmed_{R}		
	\mathfrak{S}_{\alpha,R}.
\end{align*}
Combining this with Theorems \ref{Thm-main-Mahonian-Rlwmin}--\ref{Thm-main-r-Euler--Mahonian-Rlwmin},
we obtain the following corollary.
\begin{corollary}\label{Corollary-weak-Stirling-(Euler-)Mahonian}
For any positive integer $r$, 
the statistics within each of the following three classes are equidistributed on $\mathfrak{S}_{\alpha}$:

{1.}
$(\emph{\textsf{inv}},\emph{\textsf{rlwmin}})$,
$(\emph{\textsf{maj}},\emph{\textsf{rlwmin}})$,
$(\emph{\textsf{den}},\emph{\textsf{rlwmin}})$,
$(\emph{\textsf{mak}},\emph{\textsf{rlwmin}})$,
$(\emph{\textsf{mad}},\emph{\textsf{rlwmin}})$,
$(\emph{\textsf{inv}}_{r},\emph{\textsf{rlwmin}})$,
$(r\emph{\textsf{maj}},\emph{\textsf{rlwmin}})$, and
$(r\emph{\textsf{den}},\emph{\textsf{rlwmin}})$;

{2.}
$(\emph{\textsf{rlwmin}},\emph{\textsf{des}},\emph{\textsf{maj}})$,
$(\emph{\textsf{rlwmin}},\emph{\textsf{exc}},\emph{\textsf{den}})$, and 
$(\emph{\textsf{rlwmin}},\emph{\textsf{des}},\emph{\textsf{mak}})$;

{3.}
$(\emph{\textsf{rlwmin}},r\emph{\textsf{des}},r\emph{\textsf{maj}})$ and
$(\emph{\textsf{rlwmin}},r\emph{\textsf{exc}},r\emph{\textsf{den}})$.
\end{corollary}

Let $\mathfrak{S}_{\alpha,D}^{\text{s}}$ be the set of words $w\in\mathfrak{S}_{\alpha}$ such that $\textsf{Rlmin}(w)=D$.
Clearly,
\begin{align*}
	\mathfrak{S}_{\alpha,D}^{\text{s}}=
	\Uplusmed_{\text{supp}(R)=D}		
	\mathfrak{S}_{\alpha,R}.	
\end{align*}
Combining this with Theorems \ref{Thm-main-Mahonian-Rlwmin}, \ref{Thm-main-Euler--Mahonian-Rlwmin}, and \ref{Thm-main-r-Euler--Mahonian-Rlwmin},
we immediately obtain  the following corollary.
\begin{corollary}\label{Corollary-main-Rlmin}
For any positive integer $r$, 
the statistics within each of the following three classes are equidistributed on $\mathfrak{S}_{\alpha,D}^{\emph{\text{s}}}$:

1. 
\emph{\textsf{inv}}, \emph{\textsf{maj}}, \emph{\textsf{den}}, \emph{\textsf{mak}}, \emph{\textsf{mad}},
$\emph{\textsf{inv}}_{r}$, $r\emph{\textsf{maj}}$, and $r$\emph{\textsf{den}}; 

2. $(\emph{\textsf{des}},\emph{\textsf{maj}})$,
$(\emph{\textsf{exc}},\emph{\textsf{den}})$,
and $(\emph{\textsf{des}},\emph{\textsf{mak}})$;

3. 
$(r\emph{\textsf{des}},r\emph{\textsf{maj}})$ and
$(r\emph{\textsf{exc}},r\emph{\textsf{den}})$.

\noindent Then, on $\mathfrak{S}_{\alpha}$, the statistics within each of the following three classes are equidistributed:

{1.}
$(\emph{\textsf{inv}},\emph{\textsf{rlmin}})$,
$(\emph{\textsf{maj}},\emph{\textsf{rlmin}})$,
$(\emph{\textsf{den}},\emph{\textsf{rlmin}})$,
$(\emph{\textsf{mak}},\emph{\textsf{rlmin}})$,
$(\emph{\textsf{mad}},\emph{\textsf{rlmin}})$,
$(\emph{\textsf{inv}}_{r},\emph{\textsf{rlmin}})$,
$(r\emph{\textsf{maj}},\emph{\textsf{rlmin}})$, and
$(r\emph{\textsf{den}},\emph{\textsf{rlmin}})$;

{2.}
$(\emph{\textsf{rlmin}},\emph{\textsf{des}},\emph{\textsf{maj}})$,
$(\emph{\textsf{rlmin}},\emph{\textsf{exc}},\emph{\textsf{den}})$, and 
$(\emph{\textsf{rlmin}},\emph{\textsf{des}},\emph{\textsf{mak}})$;

{3.}
$(\emph{\textsf{rlmin}},r\emph{\textsf{des}},r\emph{\textsf{maj}})$ and
$(\emph{\textsf{rlmin}},r\emph{\textsf{exc}},r\emph{\textsf{den}})$.
\end{corollary}

Let $\mathfrak{S}_{n}$ be the set of permutations of $[n]$, that is, $\mathfrak{S}_{n}=\mathfrak{S}_{\alpha}$ with $\alpha=(1,1,\ldots,1)$.
A pair of permutation statistics $(\textsf{st}_{1},\textsf{st}_{2})$  is called \emph{Mahonian--Stirling} if
$$\sum_{\pi\in\mathfrak{S}_{n}}q^{{\textsf{st}_{1}}(\pi)}x^{\textsf{st}_{2}(\pi)}
=x(x+q)(x+q+q^{2})\ldots(x+q+\cdots+q^{n-1}).$$
Bj\"{o}rner and Wachs \cite{Bjorner-1991} proved that  $(\textsf{inv},\textsf{rlmin})$ and
$(\textsf{maj},\textsf{rlmin})$ are Mahonian--Stirling.
Petersen \cite{Petersen-2011} obtained another Mahonian--Stirling pair  
$(\textsf{sor},\textsf{cyc})$.
Taking $\alpha=(1,1,\dots,1)$ in Corollary \ref{Corollary-main-Rlmin},  
so that $\mathfrak{S}_\alpha=\mathfrak{S}_n$, we obtain the following results.

1. The pairs
$({\textsf{inv}},{\textsf{rlmin}})$,
$({\textsf{maj}},{\textsf{rlmin}})$,
$({\textsf{den}},{\textsf{rlmin}})$,
$({\textsf{mak}},{\textsf{rlmin}})$,
$({\textsf{mad}},{\textsf{rlmin}})$,
$({\textsf{inv}}_{r},{\textsf{rlmin}})$,
$(r{\textsf{maj}},{\textsf{rlmin}})$, and
$(r{\textsf{den}},{\textsf{rlmin}})$ are Mahonian--Stirling.

2. The triples
$({\textsf{rlmin}},{\textsf{des}},{\textsf{maj}})$,
$({\textsf{rlmin}},{\textsf{exc}},{\textsf{den}})$, and 
$({\textsf{rlmin}},{\textsf{des}},{\textsf{mak}})$ are equidistributed on $\mathfrak{S}_{n}$.
This result was first obtained by Butler \cite{Butler-2023}, 
who referred to such triples of permutation statistics \emph{Stirling--Euler--Mahonian}.
In Corollary \ref{Cor-Mahonian--Stirling-D_{n,S}}, we give another type of Stirling--Euler--Mahonian statistics.

3. The triples
$({\textsf{rlmin}},r{\textsf{des}},r{\textsf{maj}})$ and
$({\textsf{rlmin}},r{\textsf{exc}},r{\textsf{den}})$ are equidistributed on $\mathfrak{S}_{n}$. 
This refines the equidistribution of
$(r{\textsf{des}},r{\textsf{maj}})$ and
$(r{\textsf{exc}},r{\textsf{den}})$ on $\mathfrak{S}_{n}$ proved in \cite{Liu-2024}.

\section{Mahonian statistics on permutations with a fixed descent set}\label{Section-restricted-permutations}
\subsection{The main results}
In this section, we focus on permutations, which are a special case of words. 
Hence, the statistics defined on words apply naturally to permutations.
Let $S=\{s_{1},s_{2},\ldots,s_{k}\}_{<}\subseteq[n-1]$,
i.e., $1\leqslant s_{1}<s_{2}<\cdots<s_{k}\leqslant n-1$.
Recall that $\mathfrak{S}_{n}$ is the set of permutations of $[n]$.
We denote
\begin{align*} 
	\mathcal{D}_{n}^{\smallsubseteq}(S)&=\{\pi\in\mathfrak{S}_{n}:\textsf{Des}(\pi)\subseteq S\},\\
	\mathcal{D}_{n}^{\smalleq}(S)&=\{\pi\in\mathfrak{S}_{n}:\textsf{Des}(\pi)= S\}.	
\end{align*}
These are the basic subsets of $\mathfrak{S}_{n}$. 
In Stanley's book \cite{Stanley-2012-book} and B\'{o}na's book \cite{Bona-2022},
the cardinalities of these sets are denoted by $\alpha_{n}(S)$ and $\beta_{n}(S)$, respectively. 
By the Principle of Inclusion-Exclusion,
\begin{align*} 
	|\mathcal{D}^{\smallsubseteq}_{n}(S)|&=\sum_{T\subseteq S}|\mathcal{D}^{\smalleq}_{n}(T)|,\\
	|\mathcal{D}^{\smalleq}_{n}(S)|&=\sum_{T\subseteq S}(-1)^{|S-T|}|\mathcal{D}^{\smallsubseteq}_{n}(T)|.
\end{align*}
The cardinalities of these two sets are given by
\begin{equation}\label{eq-MacMahon}
\begin{aligned} 
	|\mathcal{D}^{\smallsubseteq}_{n}(S)|&
	=\binom{n}{s_{1},s_{2}-s_{1},\dots,n-s_{k}},\\
	|\mathcal{D}^{\smalleq}_{n}(S)|&
	=\det\!\left(\binom{n-s_{i}}{s_{j+1}-s_{i}\,} \right)_{0 \leqslant i,j \leqslant k},
\end{aligned}
\end{equation}
with $s_{0}=0$ and $s_{k+1}=n$. 
The distributions of \textsf{inv} on these two sets are given by
\begin{equation}\label{eq-Stanley}
\begin{aligned} 
\sum_{\pi\in\mathcal{D}_{n}^{\tinysubseteq}(S)}q^{\textsf{inv}(\pi)}
&=\genfrac{[}{]}{0pt}{}{n}{s_{1},s_{2}-s_{1},\dots,n-s_{k}}_{q},
\\
\sum_{\pi\in\mathcal{D}_{n}^{\tinyeq}(S)}q^{\textsf{inv}(\pi)}
&=\det\!\left(\genfrac{[}{]}{0pt}{}{n-s_{i}}{s_{j+1}-s_{i}}_{q}\right)_{0 \leqslant i,j \leqslant k}.
\end{aligned}
\end{equation}
The formulas in (\ref{eq-MacMahon}) were first obtained by MacMahon \cite{MacMahon-1916-book}; 
see also \cite[Example 2.2.4]{Stanley-2012-book}. 
The formulas in (\ref{eq-Stanley}) are due to Stanley \cite{Stanley-1976}; 
see also \cite[Example 2.2.5]{Stanley-2012-book}.
In the above two determinant formulas, we adopt the conventions that
\[
\binom{i}{j}=0
\quad \text{and} \quad
\genfrac{[}{]}{0pt}{}{i}{j}_q=0
\qquad \text{whenever } j<0 \text{ or } j>i.
\]

The classical Mahonian statistic \textsf{imaj} on $\mathfrak{S}_{n}$ is defined as
$$\textsf{imaj}(\pi)=\textsf{maj}(\pi^{-1}),$$ 
where $\pi^{-1}$ is the \emph{inverse} of $\pi$, that is, 
$\pi^{-1}_{j}=i$ if and only if $\pi_{i}=j$.
Foata and Sch\"{u}tzenberger \cite{Foata-1978} proved that 
for any $S\subseteq[n-1]$, 
the statistics $\textsf{inv}$ and $\textsf{imaj}$
are equidistributed on $\mathcal{D}_{n}^{\smalleq}(S)$:
\begin{align}\label{Foata-Schutzenberger}
\sum_{\pi\in\mathcal{D}_{n}^{\tinyeq}(S)}q^{\textsf{inv}(\pi)}
=\sum_{\pi\in\mathcal{D}_{n}^{\tinyeq}(S)}q^{\textsf{imaj}(\pi)}.
\end{align}
By analogy with the definition of \textsf{imaj},
for any permutation statistic $\textsf{st}$,
define the permutation statistic $\textsf{ist}$ by
$$\textsf{ist}(\pi)=\textsf{st}(\pi^{-1}),$$
for any $\pi\in\mathfrak{S}_{n}$.

We consider the Mahonian statistics listed in Table \ref{Table-1}, excluding \textsf{\footnotesize Z}, 
since it coincides with $\mathsf{inv}$ on permutations. 
The first result of this section establishes the equidistribution of five Mahonian statistics on $D_n^{\smallsubseteq}(S)$ and $D_n^{=}(S)$.
\begin{theorem}\label{Thm-Mahonian-permutation-fix-Des}
	Let $S=\{s_{1},s_{2},\ldots,s_{k}\}_{<}\subseteq[n-1]$, 
	and set $s_{0}=0$ and $s_{k+1}=n$.
	Let $r\geqslant1$. We have
    \begin{equation}\label{Thm-Mahonian-permutation-fix-Des-eq-1}
	\begin{aligned}
    &\sum_{\pi\in\mathcal{D}_{n}^{\tinysubseteq}(S)}q^{\emph{\textsf{inv}}(\pi)}
	=\sum_{\pi\in\mathcal{D}_{n}^{\tinysubseteq}(S)}q^{\emph{\textsf{imaj}}(\pi)}
	=\sum_{\pi\in\mathcal{D}_{n}^{\tinysubseteq}(S)}q^{\emph{\textsf{imak}}(\pi)}
	=\sum_{\pi\in\mathcal{D}_{n}^{\tinysubseteq}(S)}q^{\emph{\textsf{iinv}}_{r}(\pi)}
	=\sum_{\pi\in\mathcal{D}_{n}^{\tinysubseteq}(S)}q^{\emph{\textsf{istat}}(\pi)}\\
	&=\genfrac{[}{]}{0pt}{}{n}{s_{1},s_{2}-s_{1},\dots,n-s_{k}}_{q},
	\end{aligned}
    \end{equation}
	and
	\begin{equation}\label{Thm-Mahonian-permutation-fix-Des-eq-2}
	\begin{aligned}
	&\sum_{\pi\in\mathcal{D}_{n}^{\tinyeq}(S)}q^{\emph{\textsf{inv}}(\pi)}
	=\sum_{\pi\in\mathcal{D}_{n}^{\tinyeq}(S)}q^{\emph{\textsf{imaj}}(\pi)}
	=\sum_{\pi\in\mathcal{D}_{n}^{\tinyeq}(S)}q^{\emph{\textsf{imak}}(\pi)}
	=\sum_{\pi\in\mathcal{D}_{n}^{\tinyeq}(S)}q^{\emph{\textsf{iinv}}_{r}(\pi)}
	=\sum_{\pi\in\mathcal{D}_{n}^{\tinyeq}(S)}q^{\emph{\textsf{istat}}(\pi)}\\
	&=\det\!\left(\genfrac{[}{]}{0pt}{}{n-s_{i}}{s_{j+1}-s_{i}}_{q} \right)_{0 \leqslant i,j \leqslant k}.
   \end{aligned}
   \end{equation}
\end{theorem} 
\begin{remark}\label{Remark-irmaj-imad-iden-Des}
\begin{samepage}	
None of the statistics \textsf{iden}, \textsf{imad},  \textsf{isor},  $\textsf{i}r\textsf{maj}$, or $\textsf{i}r\textsf{den}$  is equidistributed with \textsf{inv} on either $\mathcal{D}_{n}^{\smallsubseteq}(S)$ or $\mathcal{D}_{n}^{\smalleq}(S)$.
See Remark \ref{remark-reason} for an explanation.
\end{samepage}	
\end{remark}

\vspace{1.8mm}

Given a permutation $\pi=\pi_{1}\pi_{2}\ldots\pi_{n}\in\mathfrak{S}_{n}$,
a letter $\pi_{i}$ is called a \emph{left-to-right maximum} if $i=1$ or $\pi_{i}>\pi_{j}$ for every $j<i$. 
In the literature, left-to-right maxima are also referred to as \emph{records} or \emph{outstanding elements}. 
Let $\textsf{Lrmax}(\pi)$ denote the set of left-to-right maxima of $\pi$,
and let $\textsf{PLrmax}(\pi)$ denote the set of \emph{positions}  of the left-to-right maxima of $\pi$.
Let $\textsf{lrmax}(\pi)=|\textsf{Lrmax}(\pi)|=|\textsf{PLrmax}(\pi)|.$
Define
\begin{align*} 
\mathcal{D}_{n,P}^{\smallsubseteq}(S)
=\{\pi\in\mathfrak{S}_{n}:\textsf{PLrmax}(\pi)=P,~\textsf{Des}(\pi)\subseteq S\},\\
\mathcal{D}_{n,P}^{\smalleq}(S)
=\{\pi\in\mathfrak{S}_{n}:\textsf{PLrmax}(\pi)=P,~\textsf{Des}(\pi)= S\}.
\end{align*}
The second result of this section provides a refinement of the equidistribution of the first four statistics in Theorem \ref{Thm-Mahonian-permutation-fix-Des},
namely \textsf{inv}, \textsf{imaj}, \textsf{imak}, and $\textsf{iinv}_{r}$.
\begin{theorem}\label{Thm-Mahonian-permutation-fix-Des-PLrmax}
	Let $S\subseteq[n-1]$, $P\subseteq[n]$, and $r\geqslant1$. Then
	\begin{equation}\label{eq-Thm-Mahonian-permutation-fix-Des-PLrmax-1}
	\sum_{\pi\in\mathcal{D}_{n,P}^{\tinysubseteq}(S)}q^{\emph{\textsf{inv}}(\pi)}
	=\sum_{\pi\in\mathcal{D}_{n,P}^{\tinysubseteq}(S)}q^{\emph{\textsf{imaj}}(\pi)}
	=\sum_{\pi\in\mathcal{D}_{n,P}^{\tinysubseteq}(S)}q^{\emph{\textsf{imak}}(\pi)}
	=\sum_{\pi\in\mathcal{D}_{n,P}^{\tinysubseteq}(S)}q^{\emph{\textsf{iinv}}_{r}(\pi)},
    \end{equation}
	and
	\begin{equation}\label{eq-Thm-Mahonian-permutation-fix-Des-PLrmax-2}
	\sum_{\pi\in\mathcal{D}_{n,P}^{\tinyeq}(S)}q^{\emph{\textsf{inv}}(\pi)}
	=\sum_{\pi\in\mathcal{D}_{n,P}^{\tinyeq}(S)}q^{\emph{\textsf{imaj}}(\pi)}
	=\sum_{\pi\in\mathcal{D}_{n,P}^{\tinyeq}(S)}q^{\emph{\textsf{imak}}(\pi)}
	=\sum_{\pi\in\mathcal{D}_{n,P}^{\tinyeq}(S)}q^{\emph{\textsf{iinv}}_{r}(\pi)}.
    \end{equation}	
\end{theorem} 
In particular, the equidistribution of \textsf{inv} and \textsf{imaj} on $\mathcal{D}_{n,P}^{\smalleq}(S)$ is a refinement of (\ref{Foata-Schutzenberger}).
The following corollary is an immediate consequence of Theorem \ref{Thm-Mahonian-permutation-fix-Des-PLrmax}.
\begin{corollary}\label{Cor-Mahonian--Stirling-D_{n,S}}
The pairs 
$(\emph{\textsf{inv}},\emph{\textsf{lrmax}})$,
$(\emph{\textsf{imaj}},\emph{\textsf{lrmax}})$,
$(\emph{\textsf{imak}},\emph{\textsf{lrmax}})$, and
$(\emph{\textsf{iinv}}_{r},\emph{\textsf{lrmax}})$  
are equidistributed on $\mathcal{D}_{n}^{\smalleq}(S)$ and on $\mathcal{D}_{n}^{\smallsubseteq}(S)$.
\end{corollary}
As a consequence of (\ref{eq-Thm-Mahonian-permutation-fix-Des-PLrmax-2}), 
we obtain the following result.
\begin{corollary}\label{Cor-Stirling-Euler-Mahonian}
	The triples 
	$(\emph{\textsf{lrmax}},\emph{\textsf{des}},\emph{\textsf{inv}})$,
	$(\emph{\textsf{lrmax}},\emph{\textsf{des}},\emph{\textsf{imaj}})$,
	$(\emph{\textsf{lrmax}},\emph{\textsf{des}},\emph{\textsf{imak}})$,
	and
	$(\emph{\textsf{lrmax}},\emph{\textsf{des}},\emph{\textsf{iinv}}_{r})$ 
    are equidistributed on $\mathfrak{S}_{n}$.
\end{corollary}
\begin{remark}\label{remark-Stirling-Euler-Mahonian}
Corollary \ref{Cor-Stirling-Euler-Mahonian} establishes 
four Stirling--Euler--Mahonian-type triples that are equidistributed on permutations.
These are different from the Stirling--Euler--Mahonian triples
$({\textsf{rlmin}},{\textsf{des}},{\textsf{maj}})$,
$({\textsf{rlmin}},{\textsf{exc}},{\textsf{den}})$, and 
$({\textsf{rlmin}},{\textsf{des}},{\textsf{mak}})$ mentioned at the end of Section \ref{Section-restricted-words}.
Indeed, $({\textsf{lrmax}},{\textsf{des}},{\textsf{imaj}})$ and $({\textsf{rlmin}},{\textsf{des}},{\textsf{maj}})$ are not equidistributed on $\mathfrak{S}_{n}$, as can be seen from small examples.
\end{remark}

The proofs of Theorems \ref{Thm-Mahonian-permutation-fix-Des} and \ref{Thm-Mahonian-permutation-fix-Des-PLrmax} will be given in the next subsection.
In Sections \ref{Section-Alternating permutations} and \ref{Section-Permutations-with-k-runs}, we present applications of Theorems \ref{Thm-Mahonian-permutation-fix-Des} and \ref{Thm-Mahonian-permutation-fix-Des-PLrmax} to alternating permutations and permutations with $k$ alternating runs.
\subsection{Proofs of Theorems \ref{Thm-Mahonian-permutation-fix-Des} and \ref{Thm-Mahonian-permutation-fix-Des-PLrmax} }
In this subsection, let $S=\{s_{1},s_{2},\ldots,s_{k}\}_{<}\subseteq[n-1]$.
Let $\alpha_{i}=s_{i}-s_{i-1}$ for $1\leqslant i\leqslant k+1$,
with $s_{0}=0$ and $s_{k+1}=n$.
Then $\alpha=(\alpha_{1},\alpha_{2},\ldots,\alpha_{k+1})\models n$.
Clearly, $s_{i}=\sum_{j=1}^{i}\alpha_{j}$ for all $i\in[k+1]$.
For such $S$ and $\alpha$, 
we write $S=\Sigma(\alpha)$. 
For the remainder of this subsection, we always assume that
$S=\Sigma(\alpha)$. 
Then
\begin{align*} 
|\mathcal{D}_{n}^{\smallsubseteq}(S)|
=\binom{n}{s_{1},s_{2}-s_{1},\dots,n-s_{k}}
=\binom{n}{\alpha_{1},\alpha_{2},\dots,\alpha_{k+1}}
=|\mathfrak{S}_{\alpha}|.
\end{align*} 
We now define a bijection $$\theta:\mathcal{D}_{n}^{\smallsubseteq}(S)\rightarrow\mathfrak{S}_{\alpha}.$$
The construction of $\theta$ is analogous to the classical method of representing set partitions by restricted growth functions.
Let $\pi=\pi_{1}\pi_{2}\ldots\pi_{n}\in\mathcal{D}_{n}^{\smallsubseteq}(S)$.
Let 
$$\beta_{1}=\pi_{1}\pi_{2}\ldots\pi_{s_{1}},~~
\beta_{2}=\pi_{s_{1}+1}\pi_{s_{1}+2}\ldots\pi_{s_{2}},~~
\ldots,~~
\beta_{k+1}=\pi_{s_{k}+1}\pi_{s_{k}+2}\ldots\pi_{n}.$$
Thus each subsequence $\beta_{i}$ is increasing, 
since $\pi\in\mathcal{D}_{n}^{\smallsubseteq}(S)$.
Define $\theta(\pi)=w_{1}w_{2}\ldots w_{n}\in\mathfrak{S}_{\alpha}$
by setting
$$w_{i}=j,\text{~~if~the~letter~}i\text{~appears~in~}\beta_{j}.$$ 
It is straightforward to verify that $\theta:\mathcal{D}_{n}^{\smallsubseteq}(S)\rightarrow\mathfrak{S}_{\alpha}$
is a bijection.

For example, let $n=9$ and $S=\{3,5\}$.
Then $\alpha=(3,2,4)$. 
Let $\pi=245136789\in\mathcal{D}_{n}^{\smallsubseteq}(S)$.
Then
$\beta_{1}=245, \beta_{2}=13, \beta_{3}=6789$.
By definition,
$$w_{2}=w_{4}=w_{5}=1,~~w_{1}=w_{3}=2,~~w_{6}=w_{7}=w_{8}=w_{9}=3.$$
Therefore, $\theta(\pi)=212113333\in\mathfrak{S}_{(3,2,4)}$.

Before introducing a property of $\theta$, 
let us recall the \emph{standardization} map $\text{std}:\mathfrak{S}_{\alpha}\rightarrow\mathfrak{S}_{n}$.
Given a word $w\in\mathfrak{S}_{\alpha}$,
we define   
$\text{std}(w)$ to be the permutation in $\mathfrak{S}_{n}$ obtained by replacing, 
in the order of their appearance from left to right,
the $\alpha_{1}$ occurrences of the letter $1$ by $1,2,\ldots, \alpha_{1}$,
the $\alpha_{2}$ occurrences of the letter $2$ by $\alpha_{1}+1,\alpha_{1}+2,\ldots, \alpha_{1}+\alpha_{2}$, 
and similarly for the remaining letters.
For example, we have $\text{std}(313231344)=415362789$.
The following lemma is immediate.
\begin{lemma}\label{lemma-std}
	Let 
	$\operatorname{std}(w_{1}w_{2}\ldots w_{n})=\tau_{1}\tau_{2}\ldots\tau_{n}$.
	Then $\tau_{i}<\tau_{j}$ if and only if either $w_{i}< w_{j}$ or $w_{i}=w_{j}$ with $i<j$.
	In particular, for $i<j$, we have $\tau_{i}<\tau_{j}$ if and only if $w_{i}\leqslant w_{j}$.
\end{lemma}
The following property of $\theta$ will play a crucial role. 
\begin{proposition}\label{proposition-theta-pi-w}
	For $\pi\in\mathcal{D}_{n}^{\smallsubseteq}(S)$, 
	we have $\operatorname{std}(\theta(\pi))=\pi^{-1}$.
\end{proposition}
\begin{proof}
Let
$\pi=\pi_{1}\pi_{2}\ldots\pi_{n}$ and $\theta(\pi)=w=w_{1}w_{2}\ldots w_{n}$.
Fix $\ell$ with $1\leqslant \ell\leqslant k+1$, and
set $a=s_{\ell-1}+1$ and $b=s_{\ell}$.
Then 
$\beta_{\ell}=\pi_{a}\pi_{a+1}\ldots\pi_{b}, \text{~with~} \pi_{a}<\pi_{a+1}<\cdots<\pi_{b}.$
By the definition of $\theta$, we have
$$w_{\pi_{a}}=w_{\pi_{a+1}}=\cdots=w_{\pi_{b}}=\ell.$$
Let $\text{std}(w)=\tau=\tau_{1}\tau_{2}\ldots\tau_{n}$.
It follows from the definition of $\text{std}$ that
$$\tau_{\pi_{a}}=a,~\tau_{\pi_{a+1}}=a+1,~\ldots,~\tau_{\pi_{b}}=b.$$
Since $\ell$ is arbitrary, it follows that $\tau_{\pi_{i}}=i$ for all $i\in[n]$.
Hence $\tau=\pi^{-1}$.
\end{proof}
\begin{definition}
A statistic ${\textsf{st}}$ on words is said to be \emph{standard} if
${\textsf{st}}({\text{std}}(w))={\textsf{st}}(w)$ for every word $w$.
\end{definition}
\begin{proposition}\label{Prop-iS-S}
Let $\emph{\textsf{st}}$ be a standard statistic. 
Then $\emph{\textsf{ist}}(\pi)=\emph{\textsf{st}}(\theta(\pi))$ for all $\pi\in\mathcal{D}_{n}^{\smallsubseteq}(S)$.
\end{proposition}
\begin{proof}
Let $\pi\in\mathcal{D}_{n}^{\smallsubseteq}(S)$.
By definition and Proposition \ref{proposition-theta-pi-w},
we have 
$$\textsf{ist}(\pi)=\textsf{st}(\pi^{-1})
=\textsf{st}(\text{std}(\theta(\pi)))=\textsf{st}(\theta(\pi)),
$$
as desired.
\end{proof}
\begin{lemma}\label{Lemma-inv-maj-majd-mak-stat}
The statistics 
\emph{\textsf{inv}}, \emph{\textsf{maj}}, \emph{\textsf{mak}},  \emph{\textsf{stat}}, and $\emph{\textsf{inv}}_{r}$ are standard.
\end{lemma}
\begin{proof}
For any $w\in\mathfrak{S}_{\alpha}$, we have
\begin{align*} 
	\textsf{inv}(\text{std}(w))&=\textsf{inv}(w)&\!\!\!\!\!\!\!\!&\text{(by Lemma \ref{lemma-std}),}\\
	\textsf{maj}(\text{std}(w))&=\textsf{maj}(w)&\!\!\!\!\!\!\!\!&\text{(by Lemma \ref{lemma-std}),}\\
	\textsf{mak}(\text{std}(w))&=\textsf{mak}(w)&\!\!\!\!\!\!\!\!&\text{(see \cite[Section 4]{Clarke-1997}),}\\
	\textsf{stat}(\text{std}(w))&=\textsf{stat}(w)&\!\!\!\!\!\!\!\!&\text{(see \cite[(2.1)]{Fu-2019}),}\\
	\textsf{inv}_{r}(\text{std}(w))&=\textsf{inv}_{r}(w)&\!\!\!\!\!\!\!\!&\text{(see (\ref{eq-std-maj_d})),}
\end{align*}
as desired.
\end{proof}
\begin{remark}\label{remark-reason}
	We note that the statistics  
	\textsf{den}, \textsf{mad}, \textsf{sor}, $r$\textsf{maj}, and $r$\textsf{den} are not standard.
\end{remark}
\begin{proof}[Proof of Theorem \ref{Thm-Mahonian-permutation-fix-Des}]
Note that \textsf{inv}, \textsf{maj}, \textsf{mak}, 
\textsf{stat}, and  $\textsf{inv}_{r}$ are Mahonian statistics on $\mathfrak{S}_{\alpha}$.
Hence, for each of them,
the generating function on $\mathfrak{S}_{\alpha}$ is given by the $q$-multinomial coefficient $\genfrac{[}{]}{0pt}{}{n}{\alpha_{1},\alpha_{2},\dots,\alpha_{k+1}}_{q}$.
Let $\textsf{st}$ be any of \textsf{inv}, \textsf{maj}, \textsf{mak}, \textsf{stat}, and $\textsf{inv}_{r}$.	
For $\textsf{st}=\textsf{inv}$, 
note that $\textsf{ist}=\textsf{iinv}=\textsf{inv}$, since
$\textsf{inv}(\pi)=\textsf{inv}(\pi^{-1})$ for any permutation $\pi$.
By Proposition \ref{Prop-iS-S} and Lemma \ref{Lemma-inv-maj-majd-mak-stat},
and since $\theta:\mathcal{D}_{n}^{\smallsubseteq}(S)\rightarrow\mathfrak{S}_{\alpha}$ is a bijection, 
we have
\begin{align*}
 \sum_{\pi\in\mathcal{D}_{n}^{\tinysubseteq}(S)}q^{\textsf{ist}(\pi)}
=\sum_{\pi\in\mathcal{D}_{n}^{\tinysubseteq}(S)}q^{\textsf{st}(\theta(\pi))}
=\sum_{w\in\mathfrak{S}_{\alpha}}q^{\textsf{st}(w)}
=\genfrac{[}{]}{0pt}{}{n}{\alpha_{1},\alpha_{2},\dots,\alpha_{k+1}}_{q}
=\genfrac{[}{]}{0pt}{}{n}{s_{1},s_{2}-s_{1},\dots,n-s_{k}}_{q}.
\end{align*} 	
This proves (\ref{Thm-Mahonian-permutation-fix-Des-eq-1}).
Using the Principle of Inclusion-Exclusion 
(e.g., \cite[Example 2.2.5]{Stanley-2012-book}),
we obtain (\ref{Thm-Mahonian-permutation-fix-Des-eq-2}).
\end{proof}

In the remainder of this subsection, we prove Theorem \ref{Thm-Mahonian-permutation-fix-Des-PLrmax}.
Recall that ${\alpha=(\alpha_{1},\ldots,\alpha_{k+1})\models n}$. 
Define the map $$\text{istd}_{\alpha}:[n]\rightarrow[k+1]$$ as follows.
Set 
$\text{istd}_{\alpha}(i)=1$ for $1\leqslant i\leqslant \alpha_{1}$,
$\text{istd}_{\alpha}(i)=2$
for $\alpha_{1}+1\leqslant i\leqslant \alpha_{1}+\alpha_{2}$, and so on.
For a permutation $\pi\in\mathfrak{S}_{n}$, 
define $$\text{istd}_{\alpha}(\pi)=\text{istd}_{\alpha}(\pi_{1})~\!\text{istd}_{\alpha}(\pi_{2})~\!\ldots~\!\text{istd}_{\alpha}(\pi_{n}).$$
Clearly,
$\text{istd}_{\alpha}(\pi)$ is the word in $\mathfrak{S}_{\alpha}$ obtained from $\pi$ by replacing each of $1,2,\ldots,\alpha_{1}$ with $1$,
each of $\alpha_{1}+1,\alpha_{1}+2,\ldots,\alpha_{1}+\alpha_{2}$ with $2$, and so on.
It is easy to see that for any $w\in\mathfrak{S}_{\alpha}$, 
$$w=\text{istd}_{\alpha}(\text{std}(w)).$$
Moreover, for $A\subseteq[n]$, 
we define the \emph{multiset} $\text{istd}_{\alpha}(A)$ by  
$$\text{istd}_{\alpha}(A)=\{\text{istd}_{\alpha}(a):a\in A\}.$$

Let $S=\Sigma(\alpha)$, as defined at the beginning of this subsection.
A subset $P\subseteq [n]$ is said to be \emph{$S$-suffix-closed} if, whenever
$i\in P$ and $s_{j-1}<i\leqslant s_j$ for some $j\in [k+1]$, we have
$[i,s_j]:=\{i,i+1,\ldots,s_j\}\subseteq P.$
Let $2^{[n]}_{S}$ denote the collection of all $S$-suffix-closed subsets of $[n]$.
Let $2^{[k+1]}_{\alpha}$ denote the collection of all submultisets of the multiset
$\{1^{\alpha_{1}},2^{\alpha_{2}},\ldots,(k+1)^{\alpha_{k+1}}\}$.
Clearly,  $$|2^{[n]}_{S}|=|2^{[k+1]}_{\alpha}|=\prod_{j=1}^{k+1}(\alpha_j+1).$$
It is straightforward to verify the following result.
\begin{lemma}\label{istd_inverse}
The map $$\operatorname{istd}_{\alpha}:2^{[n]}_{S}\rightarrow 2^{[k+1]}_{\alpha}$$ is a bijiection.
\end{lemma}
Let $\pi\in\mathcal{D}_{n}^{\smallsubseteq}(S)$,
and let $\textsf{PLrmax}(\pi)=P$.
Since each subsequence $\beta_{j}=\pi_{s_{j-1}+1}\pi_{s_{j-1}+2}\ldots\pi_{s_{j}}$ is increasing,
it follows that $P$ is $S$-suffix-closed.
\begin{lemma}\label{lemma-P-istd(P)}
	Let $P\in2^{[n]}_{S}$, and let $R=\operatorname{istd}_{\alpha}(P)$.
	Then $\theta:\mathcal{D}_{n,P}^{\smallsubseteq}(S)\rightarrow\mathfrak{S}_{\alpha,R}$ is a bijection.
\end{lemma}
\begin{proof}
	Since
	$
	\theta:D_{n}^{\smallsubseteq}(S)\rightarrow\mathfrak{S}_{\alpha}
	$
	is a bijiection, it suffices to show that
	\[
   \theta\bigl(D_{n,P}^{\smallsubseteq}(S)\bigr)=\mathfrak{S}_{\alpha,R}.
   \]	
	
	Let $\pi\in D_n^{\smallsubseteq}(S)$, and let $w=\theta(\pi)$. By Proposition \ref{proposition-theta-pi-w}, we have
	\[
	\mathrm{std}(w)=\pi^{-1}.
	\]
	Moreover, by Lemma \ref{lemma-std},
	\[
	\mathsf{Rlwmin}(w)=\mathrm{istd}_{\alpha}\bigl(\mathsf{Rlmin}(\mathrm{std}(w))\bigr).
	\]
	Hence
	\[
	\mathsf{Rlwmin}(w)
	=\mathrm{istd}_{\alpha}\bigl(\mathsf{Rlmin}(\pi^{-1})\bigr).
	\]
	It is straightforward to verify that
	\[
	\mathsf{PLrmax}(\pi)=\mathsf{Rlmin}(\pi^{-1}).
	\]
	Therefore,
	\begin{align}\label{eq-Rlwmin-istd}
	\mathsf{Rlwmin}(w)=\mathrm{istd}_{\alpha}\bigl(\mathsf{PLrmax}(\pi)\bigr).
	\end{align}
	
	Now suppose that $\pi\in D_{n,P}^{\smallsubseteq}(S)$. Then $\mathsf{PLrmax}(\pi)=P$. By (\ref{eq-Rlwmin-istd}), we have
	\[
	\mathsf{Rlwmin}(\theta(\pi))
	=\mathrm{istd}_{\alpha}(P)
	=R.
	\]
	Thus $\theta(\pi)\in \mathfrak{S}_{\alpha,R}$, and so
	\begin{align}\label{eq-subseteq}
	\theta\bigl(D_{n,P}^{\smallsubseteq}(S)\bigr)\subseteq \mathfrak{S}_{\alpha,R}.
	\end{align}
	
	Conversely, let $w\in \mathfrak{S}_{\alpha,R}$. Since
	$
	\theta: D_n^{\smallsubseteq}(S)\to \mathfrak{S}_{\alpha}
	$
	is a bijection, there exists a unique $\pi\in D_n^{\smallsubseteq}(S)$ such that $\theta(\pi)=w$. 
	By (\ref{eq-Rlwmin-istd}), we have
	\[
	R=\mathsf{Rlwmin}(w)=\mathrm{istd}_{\alpha}\bigl(\mathsf{PLrmax}(\pi)\bigr).
	\]
	Since $R=\mathrm{istd}_{\alpha}(P)$, it follows that
	\[
	\mathrm{istd}_{\alpha}\bigl(\mathsf{PLrmax}(\pi)\bigr)=\mathrm{istd}_{\alpha}(P).
	\]
	Now both $\mathsf{PLrmax}(\pi)$ and $P$ belong to $2_S^{[n]}$. 
    Hence, by Lemma \ref{istd_inverse}, the map
	\[
	\mathrm{istd}_{\alpha}:2_S^{[n]}\to 2_{\alpha}^{[k+1]}
	\]
	is injective, so
	$
	\mathsf{PLrmax}(\pi)=P.
	$
	Therefore $\pi\in D_{n,P}^{\smallsubseteq}(S)$, and thus
	\begin{align}\label{eq-supseteq}
	\mathfrak{S}_{\alpha,R}\subseteq \theta\bigl(D_{n,P}^{\smallsubseteq}(S)\bigr).
	\end{align}
	
	Combining (\ref{eq-subseteq}) and (\ref{eq-supseteq}) completes the proof.
\end{proof}

Now we are in a position to prove Theorem \ref{Thm-Mahonian-permutation-fix-Des-PLrmax}.
\begin{proof}[Proof of Theorem \ref{Thm-Mahonian-permutation-fix-Des-PLrmax}]
If $P\notin 2_S^{[n]}$, then both $D_{n,P}^{\smallsubseteq}(S)$ and $D_{n,P}^{\smalleq}(S)$ are empty, and there is nothing to prove. 
Thus we may assume that $P\in 2_S^{[n]}$.
Let $\textsf{st}$ be any of \textsf{inv}, \textsf{maj}, \textsf{mak}, and $\textsf{inv}_{r}$.
By Proposition \ref{Prop-iS-S} and Lemma \ref{lemma-P-istd(P)},
we have
\begin{align*}
 \sum_{\pi\in\mathcal{D}_{n,P}^{\tinysubseteq}(S)}q^{\textsf{ist}(\pi)}
=\sum_{\pi\in\mathcal{D}_{n,P}^{\tinysubseteq}(S)}q^{\textsf{st}(\theta(\pi))}
=\sum_{\theta(\pi)\in\mathfrak{S}_{\alpha,R}}q^{\textsf{st}(\theta(\pi))}
=\sum_{w\in\mathfrak{S}_{\alpha,R}}q^{\textsf{st}(w)}.
\end{align*}
By Theorem \ref{Thm-main-Mahonian-Rlwmin},
we see that  \textsf{inv}, \textsf{maj}, \textsf{mak}, and $\textsf{inv}_{r}$ are equidistributed on $\mathfrak{S}_{\alpha,R}$.
It follows that \textsf{inv} (=\textsf{iinv}), \textsf{imaj}, \textsf{imak},
and $\textsf{iinv}_{r}$
are equidistributed on $\mathcal{D}_{n,P}^{\smallsubseteq}(S)$.
This completes the proof of (\ref{eq-Thm-Mahonian-permutation-fix-Des-PLrmax-1}).
Then, by the Principle of Inclusion-Exclusion, we obtain (\ref{eq-Thm-Mahonian-permutation-fix-Des-PLrmax-2}).
\end{proof}

\section{Application: set partitions}\label{Section-set partitions}
A \emph{set partition} of $[n]$ is a collection $\{B_{1},B_{2},\ldots,B_{m}\}$ of subsets of $[n]$ such that $B_{i}\neq\varnothing$, $B_{i}\cap B_{j}=\varnothing$ for $i\neq j$, and $B_{1}\cup\cdots\cup B_{m}=[n]$.
There are several well-known representations of set partitions,
such as the standard representation,
the block representation,
and the canonical representation (i.e., via restricted growth functions).
Each of these representations is associated with specific statistics and enumerative properties.
We refer the reader to Mansour's book \cite{Mansour-2013}  for further details.

Throughout this paper,
we write a set partition as $B_{1}/B_{2}/\cdots/B_{m}$,
where \begin{align} \label{set-partition-max-order}
	\max B_{1}<\max B_{2}<\cdots<\max B_{m},
\end{align}  
and the elements within each block are listed in increasing order.
For instance,  the partition 
$\{\{1,3,5,7\},\{2,6\},\{4\},\{8,9\}\}$ of $[9]$  is written as 
$4/26/1357/89$.
The sequence of cardinalities $(|B_{1}|,|B_{2}|,\ldots,|B_{m}|)$ is
called the \emph{shape} of the partition $B_{1}/B_{2}/\cdots/B_{m}$.
\begin{remark}
The most common way to arrange the blocks of a set partition is to order them in the standard order, namely
$\min B_{1}<\min B_{2}<\cdots<\min B_{m}.$
Throughout this paper, however, we order the blocks of a set partition as in (\ref{set-partition-max-order}),
because ordering blocks by their maximum elements preserves the equidistribution of many Mahonian statistics; see \cite{Liu-2022}.	
\end{remark}

Recall that  $\alpha=(\alpha_{1},\alpha_{2},\ldots,\alpha_{m})\models n$.
Let
\begin{align*}
\mathcal{P}_{\alpha}&=\{\text{set~partitions~of~}[n]\text{~with~shape~}\alpha\}\\
&=\{B_{1}/B_{2}/\cdots/B_{m}:\max B_{1}<\max B_{2}<\cdots<\max B_{m},~|B_{i}|=\alpha_{i}~\text{for~all~}i\}.
\end{align*} 
In the next two subsections, we consider two representations of elements of $\mathcal{P}_{\alpha}$, namely the word representation and the permutation representation.

\subsection{The word representation of set partitions with a fixed shape}
Let $\sigma=B_{1}/B_{2}/\cdots/B_{m}\in\mathcal{P}_{\alpha}$.
The \emph{word representation} of $\sigma$,
denoted by $w(\sigma)$, is defined by
$w(\sigma)=w_{1}w_{2}\ldots w_{n}\in\mathfrak{S}_{\alpha}$ such that $w_{i}=j$ if $i$ appears in $B_{j}$. 
Mahonian statistics on set partitions under the word representation were studied in \cite{Liu-2022},
where the author referred to this representation as the ``Mahonian representation''. 
For example,  
let $\sigma=B_{1}/B_{2}/B_{3}/B_{4}=4/26/1357/89\in\mathcal{P}_{(1,2,4,2)}$.
Let $w(\sigma)=w_{1}w_{2}\ldots w_{9}$.
By definition, we have
${w_{4}=1,}$ 
${w_{2}=w_{6}=2,}$
${w_{1}=w_{3}=w_{5}=w_{7}=3,}$
${w_{8}=w_{9}=4.}$
Therefore, $w(\sigma)=323132344$.

Let $\mathcal{P}_{\alpha}^{\text{word}}$ denote the set of words in $\mathfrak{S}_{\alpha}$ corresponding to partitions of $[n]$ of shape $\alpha$ under the word representation.
Clearly,
$$w:\mathcal{P}_{\alpha}\rightarrow\mathcal{P}_{\alpha}^{\text{word}}$$
is a one-to-one correspondence.

Note that $w\in\mathcal P_\alpha^{\mathrm{word}}$ if and only if the last occurrences of $1,2,\ldots,m$ appear in that order from left to right. Equivalently, the last occurrence of each $j\in[m]$ is a weak right-to-left minimum. Consequently,
\begin{align}\label{eq-set-partition-word-representation}
	\mathcal{P}_{\alpha}^{\text{word}}=
	\Uplusmed_{\text{supp}(R)=[m]}		
	\mathfrak{S}_{\alpha,R}.		
\end{align}

Combining (\ref{eq-set-partition-word-representation}) with Theorems \ref{Thm-main-Mahonian-Rlwmin}--\ref{Thm-main-r-Euler--Mahonian-Rlwmin},
we obtain that for each fixed positive integer $r$,  
the statistics within each of the following three classes are equidistributed on $\mathcal{P}_{\alpha}^{\text{word}}$

1. Mahonian statistics:
\textsf{inv}, \textsf{maj}, \textsf{den}, \textsf{mak}, \textsf{mad}, $\textsf{inv}_{r}$, $r\textsf{maj}$,  and $r$\textsf{den};

2. Euler--Mahonian statistics:  
$(\textsf{des},\textsf{maj})$,
$(\textsf{exc},\textsf{den})$,
and $(\textsf{des},\textsf{mak})$; 

3. $r$-Euler--Mahonian statistics: 
$(r\textsf{des},r\textsf{maj})$ and
$(r\textsf{exc},r\textsf{den})$.

\noindent The equidistribution results for the above Mahonian statistics other than $r\textsf{den}$, as well as those for the above Euler--Mahonian statistics, were already established in \cite{Liu-2022}.

In fact, by combining (\ref{eq-set-partition-word-representation}) with Theorems \ref{Thm-main-Mahonian-Rlwmin}--\ref{Thm-main-r-Euler--Mahonian-Rlwmin},
we further obtain the following refinement: 
for each fixed positive integer $r$,  the statistics within each of the following three classes are equidistributed on $\mathcal{P}_{\alpha}^{\text{word}}$

1. (\textsf{inv},\textsf{rlwmin}), (\textsf{maj},\textsf{rlwmin}), (\textsf{den},\textsf{rlwmin}),
(\textsf{mak},\textsf{rlwmin}), 
(\textsf{mad},\textsf{rlwmin}), 
($\textsf{inv}_{r}$,\textsf{rlwmin}),
($r\textsf{maj}$,\textsf{rlwmin}),  
and ($r$\textsf{den},\textsf{rlwmin});

2. $(\textsf{rlwmin},\textsf{des},\textsf{maj})$,
$(\textsf{rlwmin},\textsf{exc},\textsf{den})$,
and $(\textsf{rlwmin},\textsf{des},\textsf{mak})$; 

3.
$(\textsf{rlwmin},r\textsf{des},r\textsf{maj})$ and
$(\textsf{rlwmin},r\textsf{exc},r\textsf{den})$.

\subsection{The permutation representation of set partitions with a fixed shape}
In this subsection, we introduce the \emph{permutation representation} of set partitions and
define several associated Mahonian statistics.

Let $\sigma=B_{1}/B_{2}/\cdots/B_{m}\in\mathcal{P}_{\alpha}$.
The \emph{permutation representation} of $\sigma$, 
denoted by  $\pi(\sigma)$, 
is the permutation in $\mathfrak{S}_{n}$ obtained from $\sigma$ by erasing the bars. 
For example,  
if $\sigma=4/26/1357/89\in\mathcal{P}_{(1,2,4,2)}$,
then $\pi(\sigma)=426135789$.

Let $\mathcal{P}_{\alpha}^{\text{perm}}$ denote the set of permutations in $\mathfrak{S}_{n}$ corresponding to partitions of $[n]$ of shape $\alpha$ under the permutation representation.
For a fixed composition $\alpha$, 
a set partition $\sigma$ can be recovered from $\pi(\sigma)=\pi_{1}\pi_{2}\ldots\pi_{n}$ by inserting bars after the letters $\pi_{\alpha_{1}}$, $\pi_{\alpha_{1}+\alpha_{2}},\ldots,\pi_{\alpha_{1}+\cdots+\alpha_{m-1}}$.
Consequently,
$$\pi:\mathcal{P}_{\alpha}\rightarrow\mathcal{P}_{\alpha}^{\text{perm}}$$
is a one-to-one correspondence.

Recall that  $\alpha=(\alpha_{1},\alpha_{2},\ldots,\alpha_{m})\models n$.
Let $S=\Sigma(\alpha)$, that is,
$S=\{s_{1},s_{2},\ldots,s_{m-1}\}$ with $s_{i}=\sum_{j=1}^{i}\alpha_{j}$ for all $i\in[m-1]$. 
Let $\pi\in\mathcal{P}_{\alpha}^{\text{perm}}$ be the permutation representation of the set partition ${B_{1}/B_{2}/\cdots/B_{m}\in\mathcal{P}_{\alpha}}$.
It is clear that $\pi\in\mathcal{D}_{n}^{\smallsubseteq}(S)$.
Since
$${\max B_{1}<\max B_{2}<\cdots<\max B_{m},}$$
we have $\bar S :=S\cup\{n\}\subseteq\textsf{PLrmax}(\pi)$.
Conversely, if $\pi\in D_{n,P}^{\smallsubseteq}(S)$ with $\bar S\subseteq P$, then inserting bars after the letters $\pi_{s_1},\pi_{s_2},\ldots,\pi_{s_{m-1}}$ yields a set partition in $\mathcal P_\alpha$.
Therefore,
\begin{align}\label{eq-set-partition-perms}
	\mathcal{P}_{\alpha}^{\text{perm}}=
	\Uplusmed_{P\!\!~:\!\!~\bar{S}\subseteq P}
	\mathcal{D}_{n,P}^{\smallsubseteq}(S).
\end{align}	
Combining (\ref{eq-set-partition-perms}) with Theorem \ref{Thm-Mahonian-permutation-fix-Des-PLrmax}, 
we obtain the following theorem,
which establishes four equidistributed Mahonian statistics on set partitions with a fixed shape, 
under the permutation representation.
\begin{theorem}\label{Thm-Mahonian-set-partition-perm-rep}
	Let $r\geqslant1$. We have
	\begin{align*} 
		&\sum_{\pi\in\mathcal{P}_{\alpha}^{\emph{\text{perm}}}}q^{\emph{\textsf{inv}}(\pi)}
		=\sum_{\pi\in\mathcal{P}_{\alpha}^{\emph{\text{perm}}}}q^{\emph{\textsf{imaj}}(\pi)}
		=\sum_{\pi\in\mathcal{P}_{\alpha}^{\emph{\text{perm}}}}q^{\emph{\textsf{imak}}(\pi)}
     	=\sum_{\pi\in\mathcal{P}_{\alpha}^{\emph{\text{perm}}}}q^{\emph{\textsf{iinv}}_{r}(\pi)}
	\end{align*}	
\end{theorem} 

\begin{remark}\label{Remark-irmaj-imad-iden-set-partition}
	\begin{samepage}	
		The statistics \textsf{istat}, \textsf{iden}, $\textsf{i}r\textsf{maj}$, \textsf{imad}, and \textsf{isor} are not equidistributed with \textsf{inv} on $\mathcal{P}_{\alpha}^{\text{perm}}$.
	\end{samepage}	
\end{remark}

In fact, 
by combining (\ref{eq-set-partition-perms}) with Theorem \ref{Thm-Mahonian-permutation-fix-Des-PLrmax}, we further obtain the following refinement,
which establishes four equidistributed Mahonian--Stirling-type pairs on set partitions with a fixed shape, under the permutation representation.
\begin{theorem}\label{Thm-set-partitions-Mahonian-Stirling}
	Let $r\geqslant1$. We have
	\begin{align*} 
			\begin{split}
					\sum_{\pi\in\mathcal{P}_{\alpha}^{\emph{\text{perm}}}}q^{\emph{\textsf{inv}}(\pi)}x^{\emph{\textsf{lrmax}}(\pi)}
					=\sum_{\pi\in\mathcal{P}_{\alpha}^{\emph{\text{perm}}}}q^{\emph{\textsf{imaj}}(\pi)}x^{\emph{\textsf{lrmax}}(\pi)}
					=\sum_{\pi\in\mathcal{P}_{\alpha}^{\emph{\text{perm}}}}q^{\emph{\textsf{imak}}(\pi)}x^{\emph{\textsf{lrmax}}(\pi)}
					=\sum_{\pi\in\mathcal{P}_{\alpha}^{\emph{\text{perm}}}}q^{\emph{\textsf{iinv}}_{r}(\pi)}x^{\emph{\textsf{lrmax}}(\pi)}.
				\end{split}
		\end{align*}
\end{theorem}

\section{Application: $221$-avoiding words}\label{Section-221-avoiding words}  
Given two words $w=w_{1}w_{2}\ldots w_{n}$ and  $v=v_{1}v_{2}\ldots v_{s}$ 
with $s\leqslant n$, 
we say that $w$ \emph{contains}  $v$ if there exists a \emph{subword}
$w_{i_{1}}w_{i_{2}}\ldots w_{i_{s}}$
(so $i_{1}<i_{2}<\cdots<i_{s}$  by definition of subword) that is order-isomorphic to $v$.
Otherwise, we say that $w$ \emph{avoids} $v$, or that $w$ is $v$-avoiding.
Let $\mathfrak{S}_{\alpha}(v)$ denote the set of $v$-avoiding words in $\mathfrak{S}_{\alpha}$.
Heubach and Mansour \cite{Mansour-2006} proved that 
\begin{align}\label{eq-Heubach-Mansour-1}
	|\mathfrak{S}_{\alpha}(212)|=|\mathfrak{S}_{\alpha}(221)|=\prod_{i=1}^{m-1}(\alpha_{1}+\alpha_{2}+\cdots+\alpha_{i}+1).
\end{align}

The elements of $\mathfrak{S}_{\alpha}(212)$ are called \emph{Stirling permutations} 
(also called generalized Stirling permutations in the literature), 
and we write $\mathcal{Q}_{\alpha}=\mathfrak{S}_{\alpha}(212)$.
Stirling permutations were first introduced and studied by Gessel and Stanley \cite{Gessel-1978} in case of $\alpha=(2,2,\ldots,2)$ and later by Park \cite{Park-1994(1),Park-1994(2),Park-1994(3)} in case of $\alpha=(k,k,\ldots,k)$.
Stirling permutations have been widely studied over the past decades;
see Gessel's note in \cite{Gessel-2020}, 
which includes a list of articles on Stirling permutations.

Here we consider the other class, $\mathfrak{S}_{\alpha}(221)$. 
We refer to its elements as \emph{quasi-increasing words},
since every weakly increasing word
(equivalently, every $21$-avoiding word)
belongs to $\mathfrak{S}_{\alpha}(221)$. 
We write ${\mathcal{I}_{\alpha}=\mathfrak{S}_{\alpha}(221)}$.

Recall that $\alpha=(\alpha_{1},\alpha_{2},\ldots,\alpha_{m})\models n$.
Let $\alpha^{\prime}=(\alpha_{1},\alpha_{2},\ldots,\alpha_{m-1})$.
Given $w\in\mathcal{I}_{\alpha}$, 
let $w^{\prime}$ be the word obtained from $w$ by removing all occurrences of $m$.
Clearly, we have $w^{\prime}\in\mathcal{I}_{\alpha^{\prime}}$.
Conversely, 
given $w^{\prime}\in\mathcal{I}_{\alpha^{\prime}}$,
we obtain $(\alpha_{1}+\cdots+\alpha_{m-1}+1)$ words in $\mathcal{I}_{\alpha}$ by the following procedure:

1. Insert an $m$ into $w^{\prime}$ to serve as the first occurrence of $m$ in $w$;

2. Append $m^{\alpha_{m}-1}$ at the end (if $\alpha_{m}-1>0$).

\noindent For example, let $\alpha=(2,3,3)$, so $\alpha^{\prime}=(2,3)$.
Given the word $21122\in\mathcal{I}_{\alpha^{\prime}}$,
we can obtain $\alpha_{1}+\alpha_{2}+1=2+3+1=6$ words in $\mathcal{I}_{\alpha}$ as follows.

   \vskip -1.0mm

1. Insert a $\textcolor{red}{3}$ into $21122$, 
producing the following words:
$$21122\textcolor{red}{3}, 
2112\textcolor{red}{3}2,
211\textcolor{red}{3}22,
21\textcolor{red}{3}122,
2\textcolor{red}{3}1122,
\textcolor{red}{3}21122.$$

   \vskip -2.0mm

2. Append $\textcolor{blue}{33}$ at the end of each word, 
yielding the following words in $\mathcal{I}_{\alpha}$:  
$$21122\textcolor{red}{3}\textcolor{blue}{33},2112\textcolor{red}{3}2\textcolor{blue}{33},211\textcolor{red}{3}22\textcolor{blue}{33},21\textcolor{red}{3}122\textcolor{blue}{33},2\textcolor{red}{3}1122\textcolor{blue}{33},\textcolor{red}{3}21122\textcolor{blue}{33}.$$

Note that inserting the first $m$ creates exactly $t$ new inversions when it is placed so that it has $t$ letters to its right,
while appending $m^{\alpha_{m}-1}$ does not change the inversion number.
Thus, 
\begin{align}\label{eq-quasi-increasing-inv}
	\sum_{w\in\mathcal{I}_{\alpha}}q^{\textsf{inv}(w)}
	=[\alpha_{1}+\cdots+\alpha_{m-1}+1]_{q}\sum_{w^{\prime}\in\mathcal{I}_{\alpha^{\prime}}}q^{\textsf{inv}(w^{\prime})}
	=\prod_{i=1}^{m-1}[\alpha_{1}+\alpha_{2}+\cdots+\alpha_{i}+1]_{q}.
\end{align}

It is easy to see that $w\in\mathcal{I}_{\alpha}$ if and only if
for each $b$ with $2\leqslant b\leqslant m$,
all letters to the right of the second occurrence of $b$ are greater than or equal to $b$.
Consequently, the second through the $\alpha_{b}$-th occurrences of $b$ are all weak right-to-left minima.
Let $$B=\{2^{\alpha_{2}-1},3^{\alpha_{3}-1},\ldots,m^{\alpha_{m}-1}\},$$
where $i^{0}$ means that $i$ does not appear in $B$.
Thus, 
$w\in\mathcal{I}_{\alpha}$ if and only if $B\subseteq\textsf{Rlwmin}(w)$. 
Therefore,
\begin{align}\label{eq-quasi-increasing-words-and-Rlwmin}
	\mathcal{I}_{\alpha}=
	\Uplusmed_{R\!\!~:\!\!~B\subseteq R}	
	\mathfrak{S}_{\alpha,R}.
\end{align}
Combining (\ref{eq-quasi-increasing-inv}), (\ref{eq-quasi-increasing-words-and-Rlwmin}), and Theorems \ref{Thm-main-Mahonian-Rlwmin}--\ref{Thm-main-r-Euler--Mahonian-Rlwmin},
we obtain the following theorem.
\begin{theorem}\label{Mahonian-quasi-increasing-words}
Let $r$ be a positive integer. We have

\medskip
\noindent
(1) { Mahonian statistics.}
	\begin{align*}
		\begin{split}
			&\sum_{w\in\mathcal{I}_{\alpha}}q^{\emph{\textsf{inv}}(w)~}=
			\sum_{w\in\mathcal{I}_{\alpha}}q^{\emph{\textsf{maj}}(w)~}
			=\sum_{w\in\mathcal{I}_{\alpha}}q^{\emph{\textsf{den}}(w)~}
			=\sum_{w\in\mathcal{I}_{\alpha}}q^{\emph{\textsf{mak}}(w)}
			=\sum_{w\in\mathcal{I}_{\alpha}}q^{\emph{\textsf{mad}}(w)}\\
			=&\sum_{w\in\mathcal{I}_{\alpha}}q^{\emph{\textsf{inv}}_{r}(w)}
			=\sum_{w\in\mathcal{I}_{\alpha}}q^{r\emph{\textsf{maj}}(w)}
  		    =\sum_{w\in\mathcal{I}_{\alpha}}q^{r\emph{\textsf{den}}(w)}
			=\prod_{i=1}^{m-1}[\alpha_{1}+\alpha_{2}+\cdots+\alpha_{i}+1]_{q},
		\end{split}
	\end{align*}
	
\medskip
\noindent
(2) { Euler--Mahonian statistics.}
\begin{align*}
\sum_{w\in\mathcal{I}_{\alpha}}t^{\emph{\textsf{des}}(w)}q^{\emph{\textsf{maj}}(w)}=
\sum_{w\in\mathcal{I}_{\alpha}}t^{\emph{\textsf{exc}}(w)}q^{\emph{\textsf{den}}(w)}=
\sum_{w\in\mathcal{I}_{\alpha}}t^{\emph{\textsf{des}}(w)}q^{\emph{\textsf{mak}}(w)},
\end{align*}

\medskip
\noindent
(3) { $r$-Euler--Mahonian statistics.}
\begin{align*}
	\sum_{w\in\mathcal{I}_{\alpha}}t^{r\emph{\textsf{des}}(w)}q^{r\emph{\textsf{maj}}(w)}=
	\sum_{w\in\mathcal{I}_{\alpha}}t^{r\emph{\textsf{exc}}(w)}q^{r\emph{\textsf{den}}(w)}.
\end{align*}
\end{theorem}

The statistics \textsf{inv} and \textsf{maj} on \emph{$k$-Stirling permutations}, that is, $\mathcal{Q}_{\alpha}$ with $\alpha=(k,k,\dots,k)$, 
were studied combinatorially by Park \cite{Park-1994(1)}.
In general, \textsf{inv} and \textsf{maj} are not equidistributed on $\mathcal{Q}_\alpha$.
Moreover, \textsf{inv} does not have the same distribution on $\mathcal{Q}_\alpha$ and $\mathcal{I}_\alpha$, 
and neither does \textsf{maj}.
In what follows, we prove that 
\textsf{des}, \textsf{lrmin}, and \textsf{lrmax} each have the same distribution on these two classes.
Here  ${\textsf{lrmin}}$  denotes the number of left-to-right minima and defined by
\begin{align*}
\textsf{lrmin}(w)=|\{i:w_{i}<w_{j},\text{~for~all~}j<i\}|.
\end{align*}
\begin{theorem} \label{Thm-Stirling-quasi-increasing}
	We have
	\begin{align}
		\sum_{w\in\mathcal{Q}_{\alpha}}q^{\emph{\textsf{des}}(w)~\!~}
		&=\sum_{w\in\mathcal{I}_{\alpha}}q^{\emph{\textsf{des}}(w)},
		\label{eq-Thm-Stirling-quasi-increasing-1}\\
		\sum_{w\in\mathcal{Q}_{\alpha}}q^{\emph{\textsf{lrmin}}(w)}
		&=
		\sum_{w\in\mathcal{I}_{\alpha}}q^{\emph{\textsf{lrmin}}(w)},
		\label{eq-Thm-Stirling-quasi-increasing-2}\\
		\sum_{w\in\mathcal{Q}_{\alpha}}q^{\emph{\textsf{lrmax}}(w)}
		&=
        \sum_{w\in\mathcal{I}_{\alpha}}q^{\emph{\textsf{lrmax}}(w)}.
        \label{eq-Thm-Stirling-quasi-increasing-3}
	\end{align}
\end{theorem}
\begin{remark}
By Theorem \ref{Mahonian-quasi-increasing-words}, 
the statistics \textsf{des} and \textsf{exc} are equidistributed on $\mathcal{I}_{\alpha}$,
whereas they are not equidistributed on $\mathcal{Q}_{\alpha}$.		
\end{remark}
\begin{remark}
The statistics \textsf{lrmin} and \textsf{lrmax} are not equidistributed on $\mathcal{Q}_{\alpha}$.
The statistic \textsf{lrmin} on $k$-Stirling permutations was 
studied in \cite{Park-1994(1),Kuba-2011},
while the statistic \textsf{lrmax},
which appears to be more difficult to analyze, was investigated in \cite{Kuba-2011} from a probabilistic perspective.
\end{remark}

In the remainder of this section, we prove Theorem \ref{Thm-Stirling-quasi-increasing}.
Recall that $\alpha=(\alpha_{1},\alpha_{2},\ldots,\alpha_{m})$.
The theorem is trivial when $m=1$. 
Henceforth assume $m>1$ and
set $\alpha^{\prime}=(\alpha_{1},\alpha_{2},\ldots,\alpha_{m-1})$.
\begin{proof}[Proof of (\ref {eq-Thm-Stirling-quasi-increasing-1})]
Let $A_{\alpha,i}$ denote the number of Stirling permutations in $\mathcal{Q}_{\alpha}$ with $i$ descents.
Brenti \cite[Section 6.6]{Brenti-1989} proved that
\begin{align}\label{eq-Brenti}
	A_{\alpha,i}=(\alpha_{1}+\alpha_{2}+\cdots +\alpha_{m-1}-i+1)A_{\alpha^{\prime},i-1}+(i+1)A_{\alpha^{\prime},i}.
\end{align}
Here $A_{\beta,i}=0$  unless $0\leqslant i\leqslant \ell(\beta)-1$,
where $\ell(\beta)$ denotes the number of parts of $\beta$. 
Moreover, he proved that 
$\sum_{w\in\mathcal{Q}_{\alpha}}q^{\textsf{des}(w)}$
has only real zeros for any $\alpha$.

Let $a_{\alpha,i}$ denote the number of quasi-increasing words in $\mathcal{I}_{\alpha}$ with $i$ descents.
Consider our two-step construction of words in $\mathcal{I}_{\alpha}$ from $w^{\prime}=w^{\prime}_{1}w^{\prime}_{2}\ldots w^{\prime}_{n^{\prime}}\in\mathcal{I}_{\alpha^{\prime}}$,
where $n^{\prime}=\alpha_{1}+\alpha_{2}+\cdots +\alpha_{m-1}$.
For the first step, 
inserting $m$ after $w^{\prime}_{j}$,
where $j$ is either a descent of $w^{\prime}$ or $j=n^{\prime}$,
does not change the number of descents,
whereas inserting $m$ into any other position increases the number of descents by one.
For the second step,
appending $m^{\alpha_{m}-1}$ at the end does not change the number of descents.
This leads to the following recurrence
\begin{align}\label{recursion-des-quasi-increasing word}
	a_{\alpha,i}=(\alpha_{1}+\alpha_{2}+\cdots +\alpha_{m-1}-i+1)a_{\alpha^{\prime},i-1}+(i+1)a_{\alpha^{\prime},i},	\quad 0\leqslant i\leqslant m-1,
\end{align}
Here $a_{\beta,i}=0$ unless $0\leqslant i\leqslant \ell(\beta)-1$.
Clearly, $A_{\alpha,i}$ and $a_{\alpha,i}$ agree in the initial case $m=1$.
Comparing (\ref{recursion-des-quasi-increasing word}) with (\ref{eq-Brenti}),
and using the initial case $m=1$, we obtain by induction on $m$ that $A_{\alpha,i}=a_{\alpha,i}$,
completing the proof. 
\end{proof}

\begin{proof}[Proof of (\ref {eq-Thm-Stirling-quasi-increasing-2})]
Let $B_{\alpha,i}$ denote the number of Stirling permutations in $\mathcal{Q}_{\alpha}$ with $i$ left-to-right minima.
The numbers $B_{\alpha,i}$ satisfy the following recurrence 
\begin{align}\label{eq-lrmin-Stirling}
B_{\alpha,i}=B_{\alpha^{\prime},i-1}+(\alpha_{1}+\alpha_{2}+\cdots +\alpha_{m-1})B_{\alpha^{\prime},i}.
\end{align}	
Here $B_{\beta,i}=0$ unless $1\leqslant i\leqslant \ell(\beta)$.
To see this recurrence,
note that each Stirling permutation $w\in\mathcal{Q}_{\alpha}$ can be obtained from a Stirling permutation  $w^{\prime}\in\mathcal{Q}_{\alpha^{\prime}}$ by inserting the block $m^{\alpha_{m}}$ into $w^{\prime}$.
Inserting $m^{\alpha_{m}}$ at the beginning increases the number of left-to-right minima by one,
whereas inserting $m^{\alpha_{m}}$ in other positions does not change the  number of left-to-right minima.
	
Let $b_{\alpha,i}$ denote the number of quasi-increasing words in $\mathcal{I}_{\alpha}$ with $i$ left-to-right minima.
By our two-step construction of words in $\mathcal{I}_{\alpha}$
from words in $\mathcal{I}_{\alpha^{\prime}}$, 
we see that the numbers $b_{\alpha,i}$
satisfy the same recurrence as $B_{\alpha,i}$.
Clearly, $B_{\alpha,i}$ and $b_{\alpha,i}$ agree in the initial case $m=1$.
By induction on $m$, we obtain $B_{\alpha,i}=b_{\alpha,i}$, which completes the proof.
\end{proof}

\begin{proof}[Proof of (\ref {eq-Thm-Stirling-quasi-increasing-3})]
Recall that $\textsf{PLrmax}(w)$ is the set of positions of the left-to-right maxima of $w$.
We prove the statement by recursively defining a bijection $\varphi:\mathcal{Q}_{\alpha}\rightarrow\mathcal{I}_{\alpha}$
such that 
$$\textsf{PLrmax}(w)=\textsf{PLrmax}(\varphi(w))$$ 
for all $w\in\mathcal{Q}_{\alpha}$.	
If $m=1$, then $\mathcal{Q}_{\alpha}=\mathcal{I}_{\alpha}=\{1^{\alpha_{1}}\}$,
and we take $\varphi$ to be the identity map.
Suppose a bijection $\varphi:\mathcal{Q}_{\alpha^{\prime}}\rightarrow\mathcal{I}_{\alpha^{\prime}}$
has been defined and satisfies $\textsf{PLrmax}(w^{\prime})=\textsf{PLrmax}(\varphi(w^{\prime}))$ for all $w^{\prime}\in\mathcal{Q}_{\alpha^{\prime}}$.
Let $n^{\prime}=\alpha_{1}+\cdots+\alpha_{m-1}$.
Given $w\in\mathcal{Q}_{\alpha}$, 
assume that $w$ is obtained from $w^{\prime}=w^{\prime}_{1}w^{\prime}_{2}\ldots w^{\prime}_{n^{\prime}}\in\mathcal{Q}_{\alpha^{\prime}}$ by inserting the block $m^{\alpha_{m}}$ after $w^{\prime}_{x}$, 
where $0\leqslant x\leqslant n^{\prime}$ (with $x=0$ meaning insertion at the beginning).
Let $u^{\prime}=\varphi(w^{\prime})\in\mathcal{I}_{\alpha^{\prime}}$.
Define $\varphi(w)=u$ to be the word obtained from $u^{\prime}$ by 

1. inserting $m$ after $u^{\prime}_{x}$ (with $x=0$ meaning insertion at the beginning), and 

2. appending $m^{\alpha_{m}-1}$ at the end.

\noindent A direct verification shows that
$$
\textsf{PLrmax}(w)=\textsf{PLrmax}(w^{\prime}_{1}w^{\prime}_{2}\ldots w^{\prime}_{x})\cup\{x+1\},
$$
and similarly
$$
\textsf{PLrmax}(u)=\textsf{PLrmax}(u^{\prime}_{1}u^{\prime}_{2}\ldots u^{\prime}_{x})\cup\{x+1\}.
$$
By the induction hypothesis,
we have $$\textsf{PLrmax}(w^{\prime}_{1}w^{\prime}_{2}\ldots w^{\prime}_{n^{\prime}})=\textsf{PLrmax}(u^{\prime}_{1}u^{\prime}_{2}\ldots u^{\prime}_{n^{\prime}}).$$
Intersecting both sides with $[x]$, it follows that
$$\textsf{PLrmax}(w^{\prime}_{1}w^{\prime}_{2}\ldots w^{\prime}_{x})=\textsf{PLrmax}(u^{\prime}_{1}u^{\prime}_{2}\ldots u^{\prime}_{x}).$$
Therefore, $$\textsf{PLrmax}(w)=\textsf{PLrmax}(u),$$
as desired.
\end{proof}

\section{Application: alternating permutations}\label{Section-Alternating permutations}
A permutation $\pi=\pi_{1}\pi_{2}\ldots\pi_{n}\in\mathfrak{S}_{n}$ is called \emph{alternating} if
$$\pi_{1}>\pi_{2}<\pi_{3}>\pi_{4}<\cdots,$$
equivalently, $$\textsf{Des}(\pi)=\{1,3,5,\ldots\}\cap[n-1],$$
Similarly, $\pi$ is called \emph{reverse alternating} if
$$\pi_{1}<\pi_{2}>\pi_{3}<\pi_{4}>\cdots,$$
equivalently, $$\textsf{Des}(\pi)=\{2,4,6,\ldots\}\cap[n-1].$$
We write $\text{Alt}_{n}$ (resp. $\text{Ralt}_{n}$) for the set of alternating (resp. reverse alternating) permutations in $\mathfrak{S}_{n}$.
For more information on alternating permutations, see Stanley's survey \cite{Stanley-2010}.
It is well known that the numbers of both alternating and reverse alternating permutations in $\mathfrak{S}_{n}$
are given by the \emph{Euler number} $E_{n}$.
The Euler numbers are not to be confused with the Eulerian numbers, which count permutations by number of descents.
The following fundamental enumerative property of alternating permutations is due to Andr\'{e} \cite{Andre-1879}:
\begin{align}\label{eq-sec-tan}
	\sum_{n\geqslant0}E_{n}\frac{x^{n}}{n!}=\sec{x}+\tan{x},
\end{align} 
where $E_{0}=1$.

The \emph{$q$-Euler polynomials} $E_{n}(q)$ and $E_{n}^{\ast}(q)$
are defined by 
\begin{align}\label{eq-q-Euler-numbers}
	E_{n}(q)=\sum_{\pi\in\text{Ralt}_{n}}q^{\textsf{inv}(\pi)}\text{~~and~~}
	E_{n}^{\ast}(q)=\sum_{\pi\in\text{Alt}_{n}}q^{\textsf{inv}(\pi)}.
\end{align} 
The $q$-analogues of (\ref{eq-sec-tan}), expressed in terms of $E_{n}(q)$ and $E_{n}^{\ast}(q)$, 
can be found in \mbox{\cite[Theorem 2.1]{Stanley-2010}}.

As noted by Stanley \cite{Stanley-2010}, \textsf{inv} and \textsf{imaj} are equidistributed on $\text{Alt}_{n}$ (as well as on $\text{Ralt}_{n}$).
We now exhibit further statistics sharing this distribution.
By setting $S=\{1,3,5,\ldots\}\cap[n-1]$ in (\ref{Thm-Mahonian-permutation-fix-Des-eq-2})
and invoking (\ref{eq-q-Euler-numbers}), 
we obtain 
\begin{align*}
	\sum_{\pi\in\text{Alt}_{n}}q^{\textsf{inv}(\pi)}
	=\sum_{\pi\in\text{Alt}_{n}}q^{\textsf{imaj}(\pi)}
	=\sum_{\pi\in\text{Alt}_{n}}q^{\textsf{imak}(\pi)}
	=\sum_{\pi\in\text{Alt}_{n}}q^{\textsf{iinv}_{r}(\pi)}
	=\sum_{\pi\in\text{Alt}_{n}}q^{\textsf{istat}(\pi)}
	=E_{n}^{\ast}(q).
\end{align*}
Similarly, setting $S=\{2,4,6,\ldots\}\cap[n-1]$  in (\ref{Thm-Mahonian-permutation-fix-Des-eq-2}) and invoking (\ref{eq-q-Euler-numbers}), we obtain
\begin{align*}
	\sum_{\pi\in\text{Ralt}_{n}}q^{\textsf{inv}(\pi)}
	=\sum_{\pi\in\text{Ralt}_{n}}q^{\textsf{imaj}(\pi)}
	=\sum_{\pi\in\text{Ralt}_{n}}q^{\textsf{imak}(\pi)}
	=\sum_{\pi\in\text{Ralt}_{n}}q^{\textsf{iinv}_{r}(\pi)}
	=\sum_{\pi\in\text{Ralt}_{n}}q^{\textsf{istat}(\pi)}
	=E_{n}(q).
\end{align*}	

In what follows, we refine the above equidistribution result for the statistics
\textsf{inv}, \textsf{imaj}, \textsf{imak}, and $\textsf{iinv}_{r}$. 
Given $P\subseteq[n]$, we denote
\begin{align*} 
	\text{Alt}_{n,P}
	=\{\pi\in\text{Alt}_{n} : \textsf{PLrmax}(\pi)=P\}\text{~and~}
	\text{Ralt}_{n,P}
    =\{\pi\in\text{Ralt}_{n} : \textsf{PLrmax}(\pi)=P\}.
\end{align*}
By setting $S=\{1,3,5,\ldots\}\cap[n-1]$  in (\ref{eq-Thm-Mahonian-permutation-fix-Des-PLrmax-2}), 
we obtain 
\begin{align}\label{eq-Alt-PLrmax}
\sum_{\pi\in\text{Alt}_{n,P}}q^{\textsf{inv}(\pi)}
=\sum_{\pi\in\text{Alt}_{n,P}}q^{\textsf{imaj}(\pi)}
=\sum_{\pi\in\text{Alt}_{n,P}}q^{\textsf{imak}(\pi)}
=\sum_{\pi\in\text{Alt}_{n,P}}q^{\textsf{iinv}_{r}(\pi)}.
\end{align}
Similarly, 
by setting $S=\{2,4,6,\ldots\}\cap[n-1]$ in (\ref{eq-Thm-Mahonian-permutation-fix-Des-PLrmax-2}),  
we obtain
\begin{align}\label{eq-Ralt-PLrmax}
\sum_{\pi\in\text{Ralt}_{n,P}}q^{\textsf{inv}(\pi)}
=\sum_{\pi\in\text{Ralt}_{n,P}}q^{\textsf{imaj}(\pi)}
=\sum_{\pi\in\text{Ralt}_{n,P}}q^{\textsf{imak}(\pi)}
=\sum_{\pi\in\text{Ralt}_{n,P}}q^{\textsf{iinv}_{r}(\pi)}.
\end{align}

As an immediate consequence of (\ref{eq-Alt-PLrmax}) and (\ref{eq-Ralt-PLrmax}), 
we obtain the following equidistributed Mahonian--Stirling-type pairs for alternating and reverse alternating permutations:
\begin{equation} \label{eq-Alt-Mahonian--Stirling}
	\sum_{\pi\in\text{Alt}_{n}}q^{\textsf{inv}(\pi)}x^{\textsf{lrmax}(\pi)}
	=\sum_{\pi\in\text{Alt}_{n}}q^{\textsf{imaj}(\pi)}x^{\textsf{lrmax}(\pi)}
	=\sum_{\pi\in\text{Alt}_{n}}q^{\textsf{imak}(\pi)}x^{\textsf{lrmax}(\pi)}
	=\sum_{\pi\in\text{Alt}_{n}}q^{\textsf{iinv}_{r}(\pi)}x^{\textsf{lrmax}(\pi)}
\end{equation}	
and
\begin{equation} \label{eq-Ralt-Mahonian--Stirling}
	\sum_{\pi\in\text{Ralt}_{n}}q^{\textsf{inv}(\pi)}x^{\textsf{lrmax}(\pi)}
	=\sum_{\pi\in\text{Ralt}_{n}}q^{\textsf{imaj}(\pi)}x^{\textsf{lrmax}(\pi)}
	=\sum_{\pi\in\text{Ralt}_{n}}q^{\textsf{imak}(\pi)}x^{\textsf{lrmax}(\pi)}
	=\sum_{\pi\in\text{Ralt}_{n}}q^{\textsf{iinv}_{r}(\pi)}x^{\textsf{lrmax}(\pi)}.
\end{equation}	

\begin{remark}
For results on the distribution of $\textsf{inv}$ on $\text{Alt}_{n}$ and $\text{Ralt}_{n}$,
see \cite[Theorem 2.1]{Stanley-2010}. 
For results on the distribution of $\textsf{lrmax}$ on $\text{Alt}_{n}$ and $\text{Ralt}_{n}$,
see Carlitz and Scoville \cite{Carlitz-Scoville-1975} and 
Han, Kitaev, and Zhang \cite{Han-Kitaev-Zhang-2025}. 
In view of (\ref{eq-Alt-Mahonian--Stirling}) and (\ref{eq-Ralt-Mahonian--Stirling}),
it may be of interest to study the joint distribution of ${\textsf{inv}}$ and  ${\textsf{lrmax}}$ on $\text{Alt}_{n}$ and $\text{Ralt}_{n}$.	
\end{remark}

\section{Application: permutations with $k$ alternating runs}\label{Section-Permutations-with-k-runs}
Let $\pi=\pi_{1}\pi_{2}\ldots\pi_{n}\in\mathfrak{S}_{n}$.
We say that $\pi$ changes direction at position $i$ if either 
$\pi_{i-1}<\pi_{i}>\pi_{i+1}$
or 
$\pi_{i-1}>\pi_{i}<\pi_{i+1}$,
where $2\leqslant i\leqslant n-1$.
We say that $\pi$ has \emph{$k$ alternating runs} if there are exactly $k-1$ positions at which $\pi$ changes direction. 
Let $\mathcal{R}(n,k)$ denote the set of permutations in $\mathfrak{S}_{n}$ with $k$ alternating runs,
and let $r(n,k)=|\mathcal{R}(n,k)|$.
A substantial literature is devoted to the numbers $r(n,k)$; 
see, e.g., 
Andr\'{e} \cite{Andre-1884},
B\'{o}na--Ehrenborg \cite{Bona-2000},
Canfield--Wilf \cite{Canfield-2008},
Carlitz \cite{Carlitz-1978,Carlitz-1980,Carlitz-1981},
Ma \cite{Ma-2012},
Stanley \cite{Stanley-2008},
and Wilf \cite{Wilf-1998}.
In this section,
we investigate Mahonian statistics on the set $\mathcal{R}(n,k)$.

If $i\leqslant j$,
we denote by $[i,j]$ the set $\{i,i+1,\ldots,j\}$.
Let $S\subseteq[n-1]$ and let $i\leqslant j$.
If $[i,j]\subseteq S$ and $i-1,j+1\notin S$,
we say that $[i,j]$ is a \emph{maximal consecutive segment} of $S$.
Let $\textsf{mcs}(S)$ denote the number of \textbf{m}aximal  \textbf{c}onsecutive  \textbf{s}egments of $S$.
For example, $\textsf{mcs}(\{\underline{1,2,3},\underline{5,6},\underline{8,9,10}\})=3$.

For even $k$, let 
$$\mathcal{S}^{\text{even}}_{n,k}=\{S\subseteq[n-1]:\textsf{mcs}(S)=k/2,
\text{~exactly~one~of~}1\text{~and~}n-1\text{~lies~in~}S\}.$$
For odd $k$, let  $\mathcal{S}^{\text{odd}}_{n,k}=\mathcal{S}^{\text{odd}}_{n,k,1}\cup\mathcal{S}^{\text{odd}}_{n,k,2}$, where
\begin{align*} 
	\mathcal{S}^{\text{odd}}_{n,k,1}&=\{S\subseteq[n-1]:\textsf{mcs}(S)=(k-1)/2,~
	1\notin S,~n-1\notin S\},\\
	\mathcal{S}^{\text{odd}}_{n,k,2}&=\{S\subseteq[n-1]:\textsf{mcs}(S)=(k+1)/2,~
	1\in S,~n-1\in S\}.
\end{align*}
One checks that
\begin{align} \label{eq-runs}
	\mathcal{R}(n,k)=\left\{
	\begin{aligned}
		&\underset{S\in\mathcal{S}^{\text{even}}_{n,k}}\Uplusmed	
		\mathcal{D}_{n}^{\smalleq}(S),~~\text{if~}k\text{~is~even},\\
		&\underset{S\in\mathcal{S}^{\text{odd}}_{n,k}}\Uplusmed	
		\mathcal{D}_{n}^{\smalleq}(S),\!~~~\text{if~}k\text{~is~odd}.
	\end{aligned}
	\right.
\end{align}

Combining (\ref{eq-runs}) with (\ref{Thm-Mahonian-permutation-fix-Des-eq-2}),
we obtain the following five equidistributed Mahonian statistics on $\mathcal{R}(n,k)$:
\begin{align*}
     \sum_{\pi\in\mathcal{R}(n,k)}q^{\textsf{inv}(\pi)}
	=\sum_{\pi\in\mathcal{R}(n,k)}q^{\textsf{imaj}(\pi)}
	=\sum_{\pi\in\mathcal{R}(n,k)}q^{\textsf{imak}(\pi)}
	=\sum_{\pi\in\mathcal{R}(n,k)}q^{\textsf{iinv}_{r}(\pi)}
	=\sum_{\pi\in\mathcal{R}(n,k)}q^{\textsf{istat}(\pi)}.
\end{align*}

Let
$$\mathcal{R}(n,k;P)=\{\pi\in\mathcal{R}(n,k):\textsf{PLrmax}(\pi)=P\}.$$ 
By (\ref{eq-runs}), we have
\begin{align*}  
	\mathcal{R}(n,k;P)=\left\{
	\begin{aligned}
		&\underset{S\in\mathcal{S}^{\text{even}}_{n,k}}\Uplusmed	
		\mathcal{D}_{n,P}^{\smalleq}(S),~~~\text{if~}k\text{~is~even},\\
		&\underset{S\in\mathcal{S}^{\text{odd}}_{n,k}}\Uplusmed	
		\mathcal{D}_{n,P}^{\smalleq}(S),\!~~~~\text{if~}k\text{~is~odd}.
	\end{aligned}
	\right.
\end{align*}
Combining this with (\ref{eq-Thm-Mahonian-permutation-fix-Des-PLrmax-2}),
we obtain 
\begin{align} \label{eq-k-runs-PLrmax}
	\sum_{\pi\in\mathcal{R}(n,k;P)}q^{\textsf{inv}(\pi)}
	=\sum_{\pi\in\mathcal{R}(n,k;P)}q^{\textsf{imaj}(\pi)}
	=\sum_{\pi\in\mathcal{R}(n,k;P)}q^{\textsf{imak}(\pi)}
	=\sum_{\pi\in\mathcal{R}(n,k;P)}q^{\textsf{iinv}_{r}(\pi)}.
\end{align}

The following consequence of (\ref{eq-k-runs-PLrmax}) gives four equidistributed Mahonian--Stirling-type pairs on $\mathcal{R}(n,k)$:
\begin{align*} 
	\sum_{\pi\in\mathcal{R}(n,k)}q^{\textsf{inv}(\pi)}x^{\textsf{lrmax}(\pi)}
	=\sum_{\pi\in\mathcal{R}(n,k)}q^{\textsf{imaj}(\pi)}x^{\textsf{lrmax}(\pi)}
	=\sum_{\pi\in\mathcal{R}(n,k)}q^{\textsf{imak}(\pi)}x^{\textsf{lrmax}(\pi)}
	=\sum_{\pi\in\mathcal{R}(n,k)}q^{\textsf{iinv}_{r}(\pi)}x^{\textsf{lrmax}(\pi)}.
\end{align*}

\section{Equidistribution of \textsf{inv} and \textsf{inv}$_{r}$  on $\mathfrak{S}_{\alpha,R}$}\label{section-inv-majd}
Throughout this section and the following two sections, we assume that $r$ is a positive integer.
Kadell \cite{Kadell-1985} introduced a family of Mahonian statistics
$\textsf{inv}_{r}$ that interpolates between \textsf{maj} and \textsf{inv}.
Given $w=w_{1}w_{2}\ldots w_{n}\in\mathfrak{S}_{\alpha}$,
define
$$
\textsf{inv}_{r}(w)=|\{(i,j):i<j<i+r,~w_{i}>w_{j}\}|+\sum_{
w_{i}>w_{i+r}}i.
$$
Clearly,
$\textsf{inv}_{1}=\textsf{maj}$ and  $\textsf{inv}_{r}=\textsf{inv}$ for $r\geqslant n-1$.

Let $w\in\mathfrak{S}_{\alpha}$ and let $\text{std}(w)=\tau$.
By Lemma \ref{lemma-std}, for $i<j$, 
we have $\tau_{i}>\tau_{j}$ if and only if $w_{i}>w_{j}$.
Therefore, 
\begin{align}\label{eq-std-maj_d}
\textsf{inv}_{r}(\text{std}(w))=\textsf{inv}_{r}(w).
\end{align}

Foata \cite{Foata-1968} introduced a well-known bijection on $\mathfrak{S}_{\alpha}$ sending \textsf{maj}  to \textsf{inv}.
Kadell \cite{Kadell-1985} constructed a bijection  sending \textsf{inv} to \textsf{inv}$_{r}$.
Liang \cite{Liang-1989} later gave a bijection sending \textsf{inv}$_{r}$ to \textsf{inv}, 
which extends Foata's bijection and is the inverse of Kadell's bijection.
We use Liang's bijection to establish the equidistribution of \textsf{inv} and \textsf{inv}$_{r}$  on $\mathfrak{S}_{\alpha,R}$.

Given a non-empty word $w=w_{1}w_{2}\ldots w_{n}$ and a letter $x$,
define an operator $J_{x}$ acting on $w$ as follows.
\begin{itemize}
	\item[$\bullet$]If $w_{n}>x$,
	write $w=u_{1}b_{1}u_{2}b_{2}\ldots u_{s}b_{s}$,
	where each $b_{j}$ is a letter greater than $x$,
	and each $u_{j}$ is a word (possibly empty),
	all of whose letters are less than or equal to $x$.
	
	\item[$\bullet$]If $w_{n}\leqslant x$,
	write $w=u_{1}b_{1}u_{2}b_{2}\ldots u_{s}b_{s}$,
	where each $b_{j}$ is a letter less than or equal to $x$,
	and each $u_{j}$ is a word (possibly empty),
	all of whose letters are greater than $x$.
\end{itemize}	
In either case,  define 
$$J_{x}(w)=b_{1}u_{1}b_{2}u_{2}\ldots b_{s}u_{s}.$$

Let $w=w_{1}w_{2}\ldots w_{n}\in\mathfrak{S}_{\alpha}$.
If $r\geqslant n$, define $F_{r}(w)=w$.
In what follows, we assume that $r<n$.
We recursively define words $\gamma_{1},\gamma_{2},\ldots,\gamma_{n}$,
where $\gamma_{i}$ is a rearrangement of $w_{1}w_{2}\ldots w_{i}$.
First define $\gamma_{i}=w_{1}w_{2}\ldots w_{i}$ for $1\leqslant i\leqslant r$.
Assume that 
$$\gamma_{i}=a_{1}a_{2}\ldots a_{i-r+1}w_{i-r+2}w_{i-r+3}\ldots w_{i}$$ 
has been defined for some $i$ with $r\leqslant i\leqslant n-1$,
where $a_{1}a_{2}\ldots a_{i-r+1}$ is a rearrangement of $w_{1}w_{2}\ldots w_{i-r+1}$.
Define
\begin{align}\label{Eq-Foata-bijection-eq-d-extension}
	\gamma_{i+1}=J_{w_{i+1}}(a_{1}a_{2}\ldots a_{i-r+1})w_{i-r+2}w_{i-r+3}\ldots w_{i}w_{i+1}.
\end{align}
Finally, set $F_{r}(w)=\gamma_{n}$.
When $r=1$,  
$F_{r}$ coincides with Foata's bijection.

\begin{theorem}[\cite{Liang-1989}]\label{Thm-F_d-bijection}
For all $r\geqslant1$, $F_{r}:\mathfrak{S}_{\alpha}\rightarrow\mathfrak{S}_{\alpha}$ is a bijection,
and for all $w\in\mathfrak{S}_{\alpha}$,
$$\emph{\textsf{inv}}_{r}(w)=\emph{\textsf{inv}}(F_{r}(w)).$$ 
\end{theorem}
Before proving that Liang’s bijection preserves the multiset of weak right-to-left minima, 
we present two lemmas. 
Given a multiset $M$ and a positive integer $x$, 
let $M\cap\lceil x\rceil$ denote the submultiset of $M$ consisting of elements not exceeding $x$. 
For example,
\begin{align*}
&\textsf{Rlwmin}(3212315354646547577)=\{1,1,3,4,4,4,5,7,7\},\text{~and~}\\
&\textsf{Rlwmin}(3212315354646547577)\cap \lceil4\rceil =\{1,1,3,4,4,4\}
\end{align*}
\begin{lemma}\label{lemma-Rlwmin}
	Let $w_{1}w_{2}\ldots w_{i}w_{i+1}$ be a sequence of positive integers.
	Then
	\begin{align*}
		\emph{\textsf{Rlwmin}}(w_{1}w_{2}\ldots w_{i}w_{i+1})=\{w_{i+1}\}\cup
		\left(\emph{\textsf{Rlwmin}}(w_{1}w_{2}\ldots w_{i})\cap \lceil w_{i+1}\rceil \right),
	\end{align*}
	where $\cup$ denotes multiset union.
\end{lemma}
\begin{proof}
	By definition, a position $t$ contributes $w_t$ to $\mathsf{Rlwmin}(w_{1}w_{2}\ldots w_{i}w_{i+1})$
	if and only if either
	
	(1) $t=i+1$, or

    (2) $t\leqslant i$ and $w_{t}\leqslant \min(w_{t+1}w_{t+2}\ldots w_{i}w_{i+1})$.	
	
\noindent	In case (1), we obtain the contribution $\{w_{i+1}\}$.
	In case (2), using
	\[
	\min\{w_{t+1}, \ldots, w_i, w_{i+1}\}
	=
	\min\bigl(\min\{w_{t+1}, \ldots, w_i\},\, w_{i+1}\bigr),
	\]
	we see that (2) is equivalent to requiring
	\[
	w_t \leqslant \min\{w_{t+1}, \ldots, w_i\}
	\quad\text{and}\quad
	w_t \leqslant w_{i+1}.
	\]
	Thus, the position $t$ contributes $w_t$ to $\mathsf{Rlwmin}(w_1 w_2 \cdots w_i)\cap \lceil w_{i+1}\rceil$.
	Therefore,
	\[
	\mathsf{Rlwmin}(w_1 \cdots w_i w_{i+1})
	=
	\{w_{i+1}\}
	\cup
	\bigl(
	\mathsf{Rlwmin}(w_1 \cdots w_i)
	\cap
	\lceil w_{i+1}\rceil
	\bigr),
	\]
	completing the proof.
\end{proof}
\begin{lemma}\label{lemma-J_{x}}
	Let $w$ be a sequence of positive integers, and
	let $x$ be a positive integer. 
	We have
	\begin{align}\label{eq-lemma-J_{x}}
		\emph{\textsf{Rlwmin}}(w)\cap\lceil x\rceil =\emph{\textsf{Rlwmin}}(J_{x}(w))\cap\lceil x\rceil.
	\end{align}
\end{lemma}
\begin{proof}
	Assume that $w = w_1 w_2 \ldots w_n$ and $J_x(w)=v=v_1 v_2 \ldots v_n$.
	By the definition of $J_x$, we may write
	$$ 
	w=u_{1}b_{1}u_{2}b_{2}\ldots u_{s}b_{s} \text{~~and~~}
	v=b_{1}u_{1}b_{2}u_{2}\ldots b_{s}u_{s}.
    $$
	\noindent
	\textbf{Case 1:~}$w_n > x$.
	Then $b_j > x$ for all $j$, and every letter in each $u_j$ is at most $x$.
    Hence all letters of $w$ not exceeding $x$ occur in the subwords
	$u_1,u_2,\dots,u_s$, and their relative order is unchanged
	when passing from $w$ to $v$.
    Since every $b_j>x$, the letters $b_j$ do not affect whether a letter not exceeding $x$ is a weak right-to-left minimum.
    Therefore, for each occurrence of a letter not exceeding $x$, its contribution to $\mathsf{Rlwmin}(w)\cap\lceil x\rceil$
    is preserved when passing from $w$ to $v$.
	Thus,  (\ref{eq-lemma-J_{x}}) holds in this case.
	
	\medskip
	\noindent
	\textbf{Case 2:~}$w_n \leqslant x$.
	Then $b_j \leqslant x$ for all $j$, and every letter in each $u_j$ is strictly greater than $x$.
	Hence all letters of $w$ not exceeding $x$ are precisely
	$b_1,b_2,\dots,b_s$, and they occur in the same order
	in both $w$ and $v$.
	Since the subwords $u_j$ contain only letters greater than $x$,
	they do not influence whether a letter not exceeding $x$
	is a weak right-to-left minimum.
	Therefore, the same argument as in Case~1 shows that
    (\ref{eq-lemma-J_{x}}) holds in this case.
\end{proof}
\begin{proposition}\label{Prop-Foata-d--keeps-Rlwmin}
Let $r\geqslant1$. 
For all $w\in\mathfrak{S}_{\alpha}$,
we have $\emph{\textsf{Rlwmin}}(w)=\emph{\textsf{Rlwmin}}(F_{r}(w))$.
\end{proposition}
\begin{proof}
	If $r\geqslant n$, then $F_r(w)=w$ by definition, and the result is immediate.
	Hence we may assume that $r<n$.
	We claim that for each $i$ with $1\leqslant i\leqslant n$, 
	we have $$\textsf{Rlwmin}(\gamma_{i})=\textsf{Rlwmin}(w_{1}w_{2}\ldots w_{i}).$$
	We use induction on $i$ to prove our claim.
	The cases $1\leqslant i \leqslant r$ are straightforward.
    Assume that the claim holds for some $i$ with $r\leqslant i \leqslant n-1$. We show that it also holds for $i+1$.
    Let $w=w_{1}w_{2}\ldots w_{n}\in\mathfrak{S}_{\alpha}$.
	Write 
	$$\gamma_{i}=a_{1}a_{2}\ldots a_{i-r+1}w_{i-r+2}w_{i-r+3}\ldots w_{i}.$$
	Denote 
	$$\beta_{i}=J_{w_{i+1}}(a_{1}a_{2}\ldots a_{i-r+1})w_{i-r+2}w_{i-r+3}\ldots w_{i}.$$ 
	Thus, $\gamma_{i+1}=\beta_{i}w_{i+1}$. 
	By Lemma \ref{lemma-J_{x}}, we have
	\begin{align}\label{eq:Rlwmin-cap-pre}
    \textsf{Rlwmin}(a_{1}a_{2}\ldots a_{i-r+1})\cap\lceil w_{i+1}\rceil=
	\textsf{Rlwmin}(J_{w_{i+1}}(a_{1}a_{2}\ldots a_{i-r+1}))\cap\lceil w_{i+1}\rceil.	 
	\end{align} 
	Consequently,
	\vspace{-4mm}
\begin{equation}\label{eq:Rlwmin-cap}
	\begin{aligned}
		&\textsf{Rlwmin}(a_{1}a_{2}\ldots a_{i-r+1}w_{i-r+2}w_{i-r+3}\ldots w_{i})
		\cap \lceil w_{i+1}\rceil \\
		={}& \textsf{Rlwmin}(
		J_{w_{i+1}}(a_{1}a_{2}\ldots a_{i-r+1})\,w_{i-r+2}w_{i-r+3}\ldots w_{i}
		)
		\cap \lceil w_{i+1}\rceil .
	\end{aligned}
\end{equation}
Indeed, after appending the common suffix
$w_{i-r+2}w_{i-r+3}\ldots w_i$,
each contribution appearing on both sides of (\ref{eq:Rlwmin-cap-pre}) is either preserved on both sides or removed on both sides, while the contributions coming from the common suffix are identical.
From (\ref{eq:Rlwmin-cap}), we see that
\begin{align}\label{eq-beta-gamma}
\textsf{Rlwmin}(\gamma_{i})\cap\lceil w_{i+1}\rceil=
	\textsf{Rlwmin}(\beta_{i})\cap\lceil w_{i+1}\rceil.
\end{align}
Therefore,
\begin{align*} 
	\textsf{Rlwmin}(\gamma_{i+1})
	&=\textsf{Rlwmin}(\beta_{i}w_{i+1})\\
	&=\{w_{i+1}\}\cup(\textsf{Rlwmin}(\beta_{i})\cap\lceil w_{i+1}\rceil)
\text{\quad\quad\quad\quad\quad(by~Lemma \ref{lemma-Rlwmin})~} \\
	&=\{w_{i+1}\}\cup(\textsf{Rlwmin}(\gamma_{i})\cap\lceil w_{i+1}\rceil)\text{\quad\quad\quad\quad\quad(by (\ref{eq-beta-gamma}))} \\
	&=\{w_{i+1}\}\cup(\textsf{Rlwmin}(w_{1}w_{2}\ldots w_{i})\cap\lceil w_{i+1}\rceil)\text{\quad(by the induction hypothesis)} \\
	&=\textsf{Rlwmin}(w_{1}w_{2}\ldots w_{i}w_{i+1})\text{\quad\quad\quad\quad\quad\quad\quad~~~(by Lemma \ref{lemma-Rlwmin})}. 
\end{align*}
The induction is now complete. 
Taking $i=n$ in our claim and recalling that $F_r(w)=\gamma_n$ completes the proof.
\end{proof}
By combining Proposition \ref{Prop-Foata-d--keeps-Rlwmin} and Theorem \ref{Thm-F_d-bijection}, 
we establish the equidistribution of $\mathsf{inv}$ and $\mathsf{inv}_r$ on $\mathfrak{S}_{\alpha,R}$.

\section{Equidistribution of \textsf{inv} and $r$\textsf{maj}  on $\mathfrak{S}_{\alpha,R}$}\label{section-inv-rmaj}
The $r$-major index was introduced by
Rawlings \cite{Rawlings-1981} for permutations
and was extended to words in \cite{Rawlings-1981-2}.
Given a word $w=w_{1}w_{2}\ldots w_{n}\in\mathfrak{S}_{\alpha}$, 
a position $i$ with $1\leqslant i\leqslant n-1$ is called an \emph{$r$-descent} of $w$ if $w_{i}\geqslant w_{i+1}+r$.
Let $r\textsf{Des}(w)$ and $r\textsf{des}(w)$ denote the set and the number of $r$-descents of $w$, respectively. That is,
\begin{align*}
	r\textsf{Des}(w) =\{i\in[n-1]:w_{i}\geqslant w_{i+1}+r\}, \text{\quad\quad} 
	r\textsf{des}(w) = |r\textsf{Des}(w)|.
\end{align*}
Let
\begin{align*}
r\textsf{Inv}(w) =\{(i,j):1\leqslant i<j\leqslant n, 
w_{i}>w_{j}>w_{i}-r\}.
\end{align*}
Define the \emph{$r$-major index} of $w$,
denoted by $r\textsf{maj}(w)$, to be
\begin{align*}
r\textsf{maj}(w)=|r\textsf{Inv}(w)|+\sum_{i\in r\textsf{Des}(w)}i.
\end{align*}
It is clear that $r\textsf{maj}$ reduces to \textsf{maj} when $r=1$ and to \textsf{inv} when $r\geqslant m$.
Hence, the statistic $r\textsf{maj}$ interpolates between \textsf{maj} and \textsf{inv}.

Rawlings \cite{Rawlings-1981} showed that for any $r\geqslant1$, $r\textsf{maj}$ is equidistributed with \textsf{inv} over permutations by constructing a bijection on permutations that sends \textsf{inv} to $r\textsf{maj}$.
In \cite{Rawlings-1981-2},
he extended this bijection to words and established the equidistribution of $\textsf{inv}$ and $r\textsf{maj}$ on $\mathfrak{S}_{\alpha}$.
We now describe this bijection.

Given $w\in\mathfrak{S}_{\alpha}$, 
let $u_{j}(i)$ denote the number of letters in $w$ that lie to the right of the $j$-th occurrence of $i$ from the left  and are strictly smaller than $i$. 
For example, if $w=2152431552$, 
then $u_{1}(5)=5$, $u_{2}(5)=1$, and  $u_{3}(5)=1$. 
Clearly,
$u_{1}(i)\geqslant u_{2}(i)\geqslant\cdots\geqslant u_{\alpha_{i}}(i)$ holds for each $i\in[m]$.
It is straightforward to verify that for $w \in \mathfrak{S}_{\alpha}$, 
the inversion number is given by
$$\textsf{inv}(w)=\sum_{i=1}^{m}\sum_{j=1}^{\alpha_{i}}u_{j}(i).$$

Rawlings' bijection $R:\mathfrak{S}_{\alpha}\rightarrow\mathfrak{S}_{\alpha}$ is defined recursively. 
Recall that $\alpha=(\alpha_{1},\alpha_{2},\ldots,\alpha_{m})$.
If $m=1$, let $R$ be the identity map. 
Henceforth assume $m>1$ and
set $\alpha^{\prime}=(\alpha_{1},\alpha_{2},\ldots,\alpha_{m-1})$.
For $w\in\mathfrak{S}_{\alpha}$, 
let $w^{\prime}\in\mathfrak{S}_{\alpha^{\prime}}$ be obtained from $w$ by removing all occurrences of $m$.
Suppose $R(w^{\prime})$ is given. 
We construct $R(w)$ by successively inserting $\alpha_{m}$ copies of $m$ into $R(w^{\prime})$. 
Let $\gamma_{j}$ denote the word obtained after inserting the first $j$ copies of $m$, with $\gamma_{0}=R(w^{\prime})$. 
Given $\gamma_{j-1}$,
we obtain $\gamma_{j}$ as follows.

$\bullet$ First, place a star at each position immediately preceding each of the $j-1$ previously inserted copies of $m$.
(Thus, for $j=1$, no positions are starred.)

$\bullet$ Next, consider the unstarred positions.
Let $d+1$ be the number of unstarred positions at which inserting 
$m$ 
does not increase the number of $r$-descents.
Label these $d+1$ positions from right to left with $0, 1, \dots, d$ (so the rightmost position is always labeled by $0$). 
The remaining unstarred positions---those that would increase the number of $r$-descents by one---are labeled from left to right with ${d+1, d+2, \dots, n^{\prime}}$, where $n^{\prime}=\alpha_{1}+\cdots+\alpha_{m-1}$.

$\bullet$ Finally, insert the $j$-th copy of $m$ into the position  labeled by $u_{j}(m)$.

For example, let $r=3$ and consider inserting the third $5$ into $\gamma_{2}=215243152$.
The potential insertion positions and their labels are shown below.
$$\includegraphics[width=6.5cm]{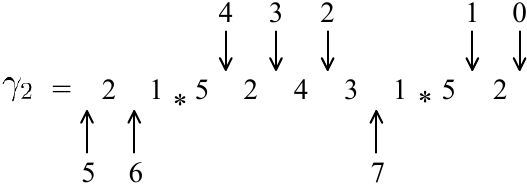}$$
The top row corresponds to positions at which the number of $3$-descents does not increase, while the bottom row corresponds to positions at which it increases by one.
If $u_{3}(5)=1$,
inserting the third $5$ into the position labeled $1$ yields the word $2152431552.$

By the above procedure and the fact that
$$u_{1}(m)\geqslant u_{2}(m)\geqslant\cdots\geqslant u_{\alpha_{m}}(m),$$
one verifies that for any $j$ with $1\leqslant j\leqslant \alpha_{m}$,
$$r\textsf{maj}(\gamma_{j})=r\textsf{maj}(\gamma_{j-1})+u_{j}(m).$$
Then by induction  we obtain
\begin{align}\label{eq-rmaj-inv}
	r\textsf{maj}(R(w))=\sum_{i=1}^{m}\sum_{j=1}^{\alpha_{i}}u_{j}(i)=
	\textsf{inv}(w).
\end{align}

We now prove that the bijection $R$ preserves the multiset of weak right-to-left minima.
\begin{proposition}\label{Prop-Rawlings-M-keeps-Rlwmin}
Let $r\geqslant1$. For any $w\in\mathfrak{S}_{\alpha}$,
we have $\emph{\textsf{Rlwmin}}(w)=\emph{\textsf{Rlwmin}}(R(w))$.
\end{proposition}
\begin{proof}
	We use induction on $m$. The case $m = 1$ is obvious.
	Below we assume that $m>1$.
	Assume that the proposition holds for $m-1$. We prove it for $m$.
	Let $w=w_{1}w_{2}\ldots w_{n}\in\mathfrak{S}_{\alpha}$,
	and let $w^{\prime}\in\mathfrak{S}_{\alpha^{\prime}}$
	be obtained from $w$ by removing all occurrences of $m$.
	Assume that the rightmost $k$ ($k\geqslant0$) letters of $w$ are all $m$, and that the $(k+1)$-st letter from the right is not $m$.
	It is clear that  
	\begin{align}\label{eq-w-last-a-Rlwmin}
		\textsf{Rlwmin}(w)=\textsf{Rlwmin}(w^{\prime})\cup\{m^{k}\}.
	\end{align}
    By the choice of $k$ and the monotonicity of $u_{j}(m)$, we have
	\begin{align}\label{eq-assume-r-u_{j}}
		u_{1}(m)\geqslant u_{2}(m)\geqslant\cdots\geqslant u_{\alpha_{m}-k}(m)> u_{\alpha_{m}-k+1}(m)=u_{\alpha_{m}-k+2}(m)=\cdots=u_{\alpha_{m}}(m)=0.
	\end{align}
	By the procedure of obtaining $\gamma_{j}$ from $\gamma_{j-1}$,
	we see that the $j$-th copy of $m$ is inserted in the rightmost position
    if and only if $u_{j}(m)=0$,
	where $\gamma_{0}=R(w^{\prime})$ and $\gamma_{\alpha_{m}}=R(w)$.
	Then by (\ref{eq-assume-r-u_{j}}), we see that
	the rightmost $k$ letters of $R(w)$ are all equal to $m$ and the  $(k+1)$-st letter of  $R(w)$ from the right is not $m$.
	Therefore,
	\begin{align}\label{eq-R(w)-Rlwmin}
		\textsf{Rlwmin}(R(w))
		=\textsf{Rlwmin}(R(w^{\prime}))\cup\{m^{k}\}.
	\end{align}
	Combining  (\ref{eq-w-last-a-Rlwmin}), (\ref{eq-R(w)-Rlwmin}),
	and the induction hypothesis that
	$\textsf{Rlwmin}(w^{\prime})=\textsf{Rlwmin}(R(w^{\prime}))$,
	we have 
	$$\textsf{Rlwmin}(w)=\textsf{Rlwmin}(R(w)).$$
	This completes the proof.
\end{proof}
Combining Proposition \ref{Prop-Rawlings-M-keeps-Rlwmin} with (\ref{eq-rmaj-inv}) yields the equidistribution of \textsf{inv} and $r\textsf{maj}$ on $\mathfrak{S}_{\alpha,R}$ for any $r\geqslant1$.

\section{Equidistribution of $(r\textsf{des},r\textsf{maj})$ and $(r\textsf{exc},r\textsf{den})$  on $\mathfrak{S}_{\alpha,R}$}\label{section-(des,maj)-(exc,den)}
The goal of this section is to establish the equidistribution of $(r\textsf{des},r\textsf{maj})$ and $(r\textsf{exc},r\textsf{den})$  on $\mathfrak{S}_{\alpha,R}$.
In particular, when $r=1$, this reduces to the equidistribution of $(\textsf{des},\textsf{maj})$ and $(\textsf{exc},\textsf{den})$ on $\mathfrak{S}_{\alpha,R}$.

Let $w=w_{1}w_{2}\ldots w_{n}\in\mathfrak{S}_{\alpha}$
with $\overline{w}=a_{1}a_{2}\ldots a_{n}$
(recall that $\overline{w}$ is the nondecreasing rearrangement of $w$). 
The \emph{two-line notation} of $w$ is 
\begin{equation*}
	w=
	\left(
	\begin{matrix}
		a_{1} & a_{2}  & a_{3}&  \ldots&a_{n}\\
		w_{1} & w_{2}  & w_{3}&  \ldots&w_{n}
	\end{matrix}
	\right).
\end{equation*}
An $r$-\emph{excedance} of $w$ is a position $i$ such that $w_{i}\geqslant a_{i}+r$.
Let $r\textsf{Exc}(w)$ and $r\textsf{exc}(w)$ denote the set and the number of $r$-excedances of $w$, respectively. 
If $i$ is an $r$-excedance of $w$,
we call $w_{i}$ an $r$-\emph{excedance-letter}.
Let $r\textsf{\scriptsize{EXCL}}(w)$
be the subsequence of $w$ that consists of the $r$-excedance-letters,
and let $r\textsf{\scriptsize{NEXCL}}(w)$ be the subsequence of $w$ that consists of the remaining letters (i.e., the non-$r$-excedance-letters).
The \emph{$r$-gap Denert's statistic}, denoted by $r\textsf{den}$,
was introduced by Huang, Lin, and Yan \cite{Huang-Lin-Yan-2026} and is defined as 
\begin{align}\label{def-rden}
 r\textsf{den}(w)=\sum_{i\in r\textsf{Exc}(w)}\left(i+B_{i}^{r}(w)\right)+\textsf{imv}(r\textsf{\scriptsize{EXCL}}(w))+\textsf{inv}(r\textsf{\scriptsize{NEXCL}}(w)),
\end{align}
where $B_{i}^{r}(w)=|\{j:w_{i}-r<a_{j}<w_{i}\}|$,
and $\textsf{imv}$ is the \emph{weakly inversion number}, defined by
$$\textsf{imv}(w)=|\{(i,j):i<j,w_{i}\geqslant w_{j}\}|.$$ 
Huang, Lin, and Yan \cite{Huang-Lin-Yan-2026} proved that for any $r\geqslant1$,
\begin{align}\label{eq-huang-lin-yan}
\sum_{w\in\mathfrak{S}_{\alpha}}t^{r\textsf{des}(w)}q^{r\textsf{maj}(w)}=
\sum_{w\in\mathfrak{S}_{\alpha}}t^{r\textsf{exc}(w)}q^{r\textsf{den}(w)}. 
\end{align} 
\begin{remark}
In the case $r=1$, we suppress the parameter $r$ in all the above notation.
Then $r\textsf{den}$ reduces to Denert's statistic for words, \textsf{den}, defined by
\begin{align}\label{def-den-Han}
\textsf{den}(w)=\sum_{i\in \textsf{Exc}(w)}i+\textsf{imv}(\textsf{\scriptsize{EXCL}}(w))+\textsf{inv}(\textsf{\scriptsize{NEXCL}}(w)),
\end{align}
which was introduced in \cite{Han-1994}.
Setting $r=1$ in (\ref{eq-huang-lin-yan}) yields the equidistribution of $(\textsf{des},\textsf{maj})$ and $(\textsf{exc},\textsf{den})$ on $\mathfrak{S}_{\alpha}$, which was proved in \cite{Han-1994}.
\end{remark}

The goal of this section is to prove the following refinement of (\ref{eq-huang-lin-yan}):
\begin{align}\label{eq-huang-lin-yan-refinement}
	\sum_{w\in\mathfrak{S}_{\alpha,R}}t^{r\textsf{des}(w)}q^{r\textsf{maj}(w)}=
	\sum_{w\in\mathfrak{S}_{\alpha,R}}t^{r\textsf{exc}(w)}q^{r\textsf{den}(w)}. 
\end{align} 
The case $r\geqslant m$ is trivial, 
since $(r\textsf{des},r\textsf{maj})$ and $(r\textsf{exc},r\textsf{den})$ both reduce to $(0,\textsf{inv})$ in this case.
Below, we assume that $1\leqslant r< m$.
To establish (\ref{eq-huang-lin-yan-refinement}) for $1\leqslant r< m$, 
it suffices to prove that the bijection $H_{r\textsf{den}}$ introduced in \cite{Huang-Lin-Yan-2026} preserves the multiset of weak right-to-left minima.
Before recalling the bijection $H_{r\textsf{den}}$, we need some notation.

A \emph{biword} $v$ is an ordered pair of multipermutations over the same multiset,
written as
\begin{equation*}
 v=\left(
 \begin{matrix}
   x_{1} & x_{2}  & x_{3}&  \ldots&x_{n}\\
   w_{1} & w_{2}  & w_{3}&  \ldots&w_{n}
  \end{matrix}
  \right).
\end{equation*}
In particular, if $x_{1}x_{2}\ldots x_{n}$ is the nondecreasing rearrangement of $w_{1}w_{2}\ldots w_{n}$, then the biword $v$ is the two-line notation of the word $w=w_{1}w_{2}\ldots w_{n}$.
Thus, the two-line notation of a word can be viewed as a biword.
\begin{definition}
A biword
$
 \left(
 \begin{matrix}
   x_{1} & x_{2}  & x_{3}&  \ldots&x_{n}\\
   w_{1} & w_{2}  & w_{3}&  \ldots&w_{n}
  \end{matrix}
  \right)
$
is called a \emph{dominated cycle} if either $n=1$ or $n>1$, $w_{1}=x_{n}$, $w_{i}=x_{i-1}$ and $w_{1}>w_{i}$ for all $2\leqslant i\leqslant n$.
\end{definition}

\begin{definition}
Let $\mathbb{P}$ be the set of positive integers.
Given $x,y\in\mathbb{P}$, the cyclic interval $\rrbracket x, y \rrbracket$ is defined by
\begin{align*}
\rrbracket x, y \rrbracket=
\begin{cases}
\{z\in \mathbb{P}:x<z\leqslant y\},& \text{if~}x\leqslant y;\\
\{z\in \mathbb{P}:x<z\text{~or~} z\leqslant y\}& \text{if~}x> y.
\end{cases}
\end{align*}
\end{definition}

\begin{definition}
Let $1\leqslant r< m$. 
For any $1\leqslant i\leqslant n-1$, define the operator $T_{i}^{r}$ on biword 
$v=
 \left(
 \begin{matrix}
   x_{1} & x_{2}  & x_{3}&  \ldots&x_{n}\\
   w_{1} & w_{2}  & w_{3}&  \ldots&w_{n}
  \end{matrix}
  \right)
$ to be
$$ T_{i}^{r}(v)=
 \left(
 \begin{matrix}
   x_{1} & x_{2}& \ldots&x_{i-1}\\
   w_{1} & w_{2}& \ldots&w_{i-1}
  \end{matrix}
  \right)
  T^{r}
  \left(
 \begin{matrix}
   x_{i} & x_{i+1} \\
   w_{i} & w_{i+1}
  \end{matrix}
  \right)
  \left(
 \begin{matrix}
   x_{i+2} &  \ldots&x_{n}\\
   w_{i+2} &  \ldots&w_{n}
  \end{matrix}
  \right),$$
  where
$$T^{r}
  \left(
 \begin{matrix}
   x & y \\
   a & b
  \end{matrix}
  \right) =
\begin{cases}
  \left(
 \begin{matrix}
   y & x \\
   b &a
  \end{matrix}
  \right),& \text{if~exactly~one~of~}a\text{~and~}b\text{~lies~in~}\rrbracket x+r-1, y+r-1 \rrbracket,\\
  \left(
 \begin{matrix}
   y & x \\
a&b
  \end{matrix}
  \right),& \text{otherwise}.
\end{cases}
  $$
\end{definition}

Now we state the algorithm $\Gamma_{r\textsf{den}}$ which decomposes a word into dominated cycles.
Let $w=w_{1}w_{2}\ldots w_{n}\in\mathfrak{S}_{\alpha}$,
and write it in the two-line notation
\begin{equation*}
 w=
 \left(
 \begin{matrix}
   a_{1} & a_{2}  & a_{3}&  \ldots&a_{n}\\
   w_{1} & w_{2}  & w_{3}&  \ldots&w_{n}
  \end{matrix}
  \right).
\end{equation*}

If $n=1$, then $w$ itself is a dominated cycle.
In what follows, we assume that $n\geqslant2$.

If $w_{n}=a_{n}$, then set
\begin{equation*}
 \widetilde{w}=
 \left(
 \begin{matrix}
   a_{1} & a_{2}  & a_{3}&  \ldots&a_{n-1}\\
   w_{1} & w_{2}  & w_{3}&  \ldots&w_{n-1}
  \end{matrix}
  \right),
 ~~u=
 \left(
 \begin{matrix}
   a_{n} \\
   w_{n}
  \end{matrix}
  \right).
\end{equation*}

Otherwise $w_{n}\neq a_{n}$, let $i_{1}$ be the largest index such that $a_{i_{1}}= w_{n}$, and set
\begin{equation*}
 w^{(1)}=T_{n-2}^{r}\circ T_{n-3}^{r}\circ\cdots \circ T_{i_{1}}^{r}(w)=
 \left(
 \begin{matrix}
   a_{1}^{(1)} & a_{2}^{(1)}  &\ldots&a_{n-2}^{(1)}& w_{n}&a_{n}\\
   w_{1}^{(1)} & w_{2}^{(1)}  & \ldots&w_{n-2}^{(1)}&w_{n-1}^{(1)} &w_{n}
  \end{matrix}
  \right).
\end{equation*}
Thus, if $i_{1}=n-1$, the above composition is empty, 
and hence $w^{(1)} = w$.

If $w_{n-1}^{(1)}=a_{n}$, then set
\begin{equation*}
\widetilde{w}=
 \left(
 \begin{matrix}
   a_{1}^{(1)} & a_{2}^{(1)}  &\ldots&a_{n-2}^{(1)}\\
   w_{1}^{(1)} & w_{2}^{(1)}  & \ldots&w_{n-2}^{(1)}
  \end{matrix}
  \right),
 ~~u=
 \left(
 \begin{matrix}
   w_{n}&a_{n} \\
   a_{n}&w_{n}
  \end{matrix}
  \right).
\end{equation*}

Otherwise $w_{n-1}^{(1)}\neq a_{n}$, let $i_{2}$ be the largest index such that $a_{i_{2}}^{(1)}= w_{n-1}^{(1)}$, and set
\begin{equation*}
 w^{(2)}=T_{n-3}^{r}\circ T_{n-4}^{r}\circ\cdots\circ T_{i_{2}}^{r}(w^{(1)})=
 \left(
 \begin{matrix}
a_{1}^{(2)} & a_{2}^{(2)}&\ldots&a_{n-3}^{(2)}& w_{n-1}^{(1)}&w_{n}&a_{n}\\
w_{1}^{(2)} & w_{2}^{(2)}& \ldots&w_{n-3}^{(2)}&w_{n-2}^{(2)}&w_{n-1}^{(1)} &w_{n}
  \end{matrix}
  \right).
\end{equation*}
Thus, if $i_{2}=n-2$, the above composition is empty, 
and hence $w^{(2)} = w^{(1)}$.

Similarly, we can repeat the above process by considering whether $w_{n-2}^{(2)}$ is equal to $a_{n}$.
So we can obtain a sequence of words $w^{(1)},w^{(2)},\ldots,w^{(t)}$  such that
\begin{equation*}
 w^{(t)}=
 \left(
 \begin{matrix}
a_{1}^{(t)} & a_{2}^{(t)}&\ldots &a_{n-t-1}^{(t)}&w_{n-t+1}^{(t-1)}&\ldots&w_{n}         &a_{n}\\
w_{1}^{(t)} & w_{2}^{(t)}& \ldots&w_{n-t-1}^{(t)}&w_{n-t}^{(t)}&\ldots&w_{n-1}^{(1)} &w_{n}
  \end{matrix}
  \right),
\end{equation*}
where $w_{n-t}^{(t)}=a_{n}$ and $w_{n-t+1}^{(t-1)}\neq a_{n},\ldots,w_{n-1}^{(1)}\neq a_{n},w_{n}\neq a_{n}$.
Set
\begin{equation*}
\widetilde{w}=
 \left(
 \begin{matrix}
a_{1}^{(t)} & a_{2}^{(t)}&\ldots &a_{n-t-1}^{(t)}\\
w_{1}^{(t)} & w_{2}^{(t)}& \ldots&w_{n-t-1}^{(t)}
  \end{matrix}
  \right),
 ~~u=
 \left(
 \begin{matrix}
w_{n-t+1}^{(t-1)}&w_{n-t+2}^{(t-2)}&\ldots&w_{n}     &a_{n}\\
a_{n}    &w_{n-t+1}^{(t-1)}&\ldots&w_{n-1}^{(1)} &w_{n}
  \end{matrix}
  \right).
\end{equation*}
Note that $a_{n}\geqslant a_{i}$ for $1\leqslant i\leqslant n$, and
$w_{n-t+1}^{(t-1)}\neq a_{n},\ldots,w_{n-1}^{(1)}\neq a_{n},w_{n}\neq a_{n}$,
then we see that $u$ is a dominated cycle.

Observe that $\widetilde{w}$ is the two-line notation of the word
$w_{1}^{(t)}w_{2}^{(t)}\ldots w_{n-t-1}^{(t)}$.
By induction, assume that $\widetilde{w}$ admits a decomposition $u_{1},u_{2},\ldots,u_{s}$ into dominated cycles.
Then the  decomposition of $w$ is defined to be $u_{1},u_{2},\ldots,u_{s},u$.
Let $\Gamma_{r\textsf{den}}(w)=u_{1}u_{2}\ldots u_{s}u$.
Finally, define $H_{r\textsf{den}}(w)$ to be the word obtained by concatenating the bottom rows of the  cycles of $\Gamma_{r\textsf{den}}(w)$.
An illustrative example of the process for $\Gamma_{r\textsf{den}}(w)$ can be found in \cite[Example 4.5]{Huang-Lin-Yan-2026}.

\begin{theorem}[Huang--Lin--Yan \cite{Huang-Lin-Yan-2026}]\label{Thm-Huang-Lin-Yan-2026}
Let $1\leqslant r< m$.
Then $H_{r\textsf{den}}:\mathfrak{S}_{\alpha}\rightarrow\mathfrak{S}_{\alpha}$ is a bijection
satisfying
$$\left(r\emph{\textsf{exc}},r\emph{\textsf{den}}\right)w=
\left(r\emph{\textsf{des}},r\emph{\textsf{maj}}\right)H_{r\textsf{den}}(w)$$ for all $w\in\mathfrak{S}_{\alpha}$.
\end{theorem}

The following proposition shows that $H_{r\textsf{den}}:\mathfrak{S}_{\alpha}\rightarrow\mathfrak{S}_{\alpha}$ preserves the multiset of weak right-to-left minima,
which implies the equidistribution of $(r\textsf{des},r\textsf{maj})$ and $(r\textsf{exc},r\textsf{den})$  on $\mathfrak{S}_{\alpha,R}$.
\begin{proposition}\label{Prop-HLY-keeps-Rlwmin}
	Let $1\leqslant r< m$. For any $w\in\mathfrak{S}_{\alpha}$,
	we have $\emph{\textsf{Rlwmin}}(w)=\emph{\textsf{Rlwmin}}(H_{r\emph{\textsf{den}}}(w))$.
\end{proposition}
\begin{proof}
Let $w\in\mathfrak{S}_{\alpha}$.
During the construction of $\Gamma_{r\textsf{den}}(w)$,
suppose that the operations 
$T_{k_{1}}^{r}$, $T_{k_{2}}^{r}$, $\ldots,$ $T_{k_{N}}^{r}$ are applied in this order.
Let
$${\beta_{1} \choose \gamma_{1}}={\overline{w} \choose w}\text{~and~}
{\beta_{i+1} \choose \gamma_{i+1}}=T_{k_{i}}^{r}{\beta_{i} \choose \gamma_{i}}\text{~for~}1\leqslant i\leqslant N.$$
Consider the bottom rows $$\gamma_{1},\gamma_{2},\gamma_{3},\ldots,\gamma_{N+1},$$
where $\gamma_{1}=w$, $\gamma_{N+1}=H_{r\textsf{den}}(w)$.
Clearly, for $1\leqslant i\leqslant N$, 
we have either $\gamma_{i+1}=\gamma_{i}$ or $\gamma_{i+1}$ is obtained from $\gamma_{i}$ by interchanging two consecutive entries.
We claim that
$$
\textsf{Rlwmin}(\gamma_{1})=\textsf{Rlwmin}(\gamma_{2})=\cdots=\textsf{Rlwmin}(\gamma_{N+1}).
$$
This will prove the proposition.
To prove this claim, it suffices to show that
\begin{align}\label{eq-Rlwmin-gamma_t}
	\textsf{Rlwmin}(\gamma_{t})=\textsf{Rlwmin}(\gamma_{t+1}),
\end{align}
for each $1\leqslant t\leqslant N$.
The case $\gamma_t=\gamma_{t+1}$ is clear.
So assume that
$$\gamma_{t}=b_{1}b_{2}\ldots b_{n}
\quad\text{and}\quad\gamma_{t+1}=b_{1}b_{2}\ldots b_{j-1}b_{j+1}b_{j}b_{j+2}\ldots b_{n}.$$
Since $\gamma_t$ and $\gamma_{t+1}$ differ only by swapping two adjacent letters,
it is enough to compare the contributions of $b_j$ and $b_{j+1}$ to the multisets
$\textsf{Rlwmin}(\gamma_t)$ and $\textsf{Rlwmin}(\gamma_{t+1})$.
By definition, 
\begin{align*}
  \left(
 \begin{matrix}
  \beta_{t+1} \\
   \gamma_{t+1}
  \end{matrix}
  \right)=T_{k_{t}}^{r}  \left(
 \begin{matrix}
  \beta_{t} \\
   \gamma_{t}
  \end{matrix}
  \right)&=
 \left(
 \begin{matrix}
   \ast & \ast& \ldots&\ast\\
   b_{1} & b_{2}& \ldots&b_{j-1}
  \end{matrix}
  \right)
  T^{r}
  \left(
 \begin{matrix}
   x & y \\
   b_{j} & b_{j+1}
  \end{matrix}
  \right)
  \left(
 \begin{matrix}
   \ast &  \ldots&\ast\\
   b_{j+2} &  \ldots&b_{n}
  \end{matrix}
  \right),\\
  &= \left(
 \begin{matrix}
   \ast & \ast& \ldots&\ast\\
   b_{1} & b_{2}& \ldots&b_{j-1}
  \end{matrix}
  \right)
  \left(
 \begin{matrix}
   y & x \\
   b_{j+1} & b_{j}
  \end{matrix}
  \right)
  \left(
 \begin{matrix}
   \ast &  \ldots&\ast\\
   b_{j+2} &  \ldots&b_{n}
  \end{matrix}
  \right),
  \end{align*}
 where $j=k_{t}$.
From the construction of $\Gamma_{r\textsf{den}}(w)$, 
we note the following two facts:

(i)~ $x< y$. Then $\rrbracket x+r-1, y+r-1 \rrbracket=\{x+r,x+r+1,\ldots,y+r-1\}$;

(ii) $x\in\{b_{j+2},b_{j+3},\ldots,b_{n}\}$.

\noindent Since $$T^{r}
\left(
\begin{matrix}
	x & y \\
	b_{j} & b_{j+1}
\end{matrix}
\right) =
\left(
\begin{matrix}
	y & x \\
	b_{j+1} & b_{j}
\end{matrix}
\right),$$ 
we see that exactly one of $b_{j}$ and $b_{j+1}$ lies in $\rrbracket x+r-1, y+r-1 \rrbracket=\{x+r,x+r+1,\ldots,y+r-1\}$.

\noindent \textbf{Case 1:} Suppose that $b_{j}\in\{x+r,\ldots,y+r-1\}$ and $b_{j+1}\notin\{x+r,\ldots,y+r-1\}$.
The conditions $b_j>x$ and $x\in\{b_{j+2},b_{j+3},\ldots,b_n\}$ imply that
the letter $b_j$ does not contribute to $\textsf{Rlwmin}(\gamma_t)$ or to $\textsf{Rlwmin}(\gamma_{t+1})$.
For $b_{j+1}$, we consider the following two possibilities:
\begin{itemize}
\item[$\bullet$]If $b_{j+1}\geqslant y+r$, then $b_{j+1}>x$.	
Since $x\in\{b_{j+2},b_{j+3},\ldots,b_n\}$, the letter $b_{j+1}$ does not contribute to
$\textsf{Rlwmin}(\gamma_t)$ or to $\textsf{Rlwmin}(\gamma_{t+1})$.
Therefore,
$
\textsf{Rlwmin}(\gamma_t)=\textsf{Rlwmin}(\gamma_{t+1}).
$	
\item[$\bullet$]If $b_{j+1}< x+r$, then $b_{j+1}<b_{j}$.	
In this case, swapping $b_j$ and $b_{j+1}$ does not affect whether the letter $b_{j+1}$
contributes to the multiset of weak right-to-left minima.
Thus,
$
\textsf{Rlwmin}(\gamma_t)=\textsf{Rlwmin}(\gamma_{t+1}).
$
\end{itemize}

\noindent \textbf{Case 2:} 
Suppose that $b_{j+1}\in\{x+r,\ldots,y+r-1\}$ and $b_j\notin\{x+r,\ldots,y+r-1\}$.
This case is symmetric to Case 1 after interchanging the roles of $b_j$ and $b_{j+1}$.
Therefore, $\textsf{Rlwmin}(\gamma_t)=\textsf{Rlwmin}(\gamma_{t+1})$.

This completes the proof of (\ref{eq-Rlwmin-gamma_t}).
\end{proof}

\section{Equidistribution of $(\textsf{des},\textsf{mak},\textsf{mad})$  and 
$(\textsf{exc},\textsf{den},\textsf{inv})$ on $\mathfrak{S}_{\alpha,R}$}\label{Section-tri}
The Mahonian permutation statistic \textsf{mak} was introduced by Foata and Zeilberger \cite{Foata-Zeilberger-1990}.
Clarke, Steingr\'{\i}msson, and Zeng \cite{Clarke-1997} extended it to words.
In the same paper, they also introduced a new Mahonian statistic  \textsf{mad} on words,
and proved that the triples $(\textsf{des},\textsf{mak},\textsf{mad})$  and 
$(\textsf{exc},\textsf{den},\textsf{inv})$ are equidistributed on
$\mathfrak{S}_{\alpha}$ by exhibiting a bijection
on $\mathfrak{S}_{\alpha}$,
which we denote by $\Phi_{\alpha}$,  
that sends $(\textsf{des},\textsf{mak},\textsf{mad})$ to $(\textsf{exc},\textsf{den},\textsf{inv})$.
The goal of this section is to establish the equidistribution of $(\textsf{des},\textsf{mak},\textsf{mad})$  and 
$(\textsf{exc},\textsf{den},\textsf{inv})$ on $\mathfrak{S}_{\alpha,R}$
by proving that the bijection $\Phi_{\alpha}$ preserves the multiset of weak right-to-left minima.

The statistic \textsf{den} is defined in (\ref{def-den-Han}). 
We now give the definitions of \textsf{mak} and \textsf{mad}.
Before doing so, we need some notions.
Let $w=w_{1}w_{2}\ldots w_{n}\in\mathfrak{S}_{\alpha}$.
The \emph{height} of a letter $a$ in $w$, denoted by $h(a)$,  
is one more than the number of letters in $w$ that are strictly smaller than $a$.
The \emph{value} of the $i$-th letter in $w$, denoted by $v_{i}$,
is defined by
$$v_{i}=h(w_{i})+l(i),$$
where $l(i)$ is the number of letters in $w$ that are to the left of $w_{i}$ and
equal to $w_{i}$.
For example, 
if $w=21144231$,
then $\overline{w}=11122344$,
so the heights of $1,2,3,4$ are, respectively, $1,4,6,7$.
The values of the letters of $w$ are given by $4,1,2,7,8,5,6,3$ in the order in which they appear in $w$.
It is easy to see that $v_{1}v_{2}\ldots v_{n}=\text{std}(w)$.

Let $w=w_{1}w_{2}\ldots w_{n}\in\mathfrak{S}_{\alpha}$.
Recall that a descent of $w$ is a position $i$ such that $w_{i}>w_{i+1}$.
If $i$ is a descent of $w$, 
then $w_{i}$ is called a \emph{descent top}, and $w_{i+1}$ a \emph{descent bottom}.
Recall also that an excedance of $w$ is a position $i$ such that $w_{i}>a_{i}$, where $\overline{w}=a_{1}a_{2}\ldots a_{n}$.  
If $i$ is an excedance of $w$, 
then $w_{i}$ is called a \emph{excedance top}, and $a_{i}$ a \emph{excedance bottom}.

The \emph{descent tops sum} of $w$, denoted by $\text{Dtop}(w)$,
is the sum of the heights of the descent tops of $w$.
The \emph{descent bottoms sum} of $w$,
denoted by $\text{Dbot}(w)$, is the sum of the values of the descent bottoms of $w$.
The \emph{descent difference} of $w$ is
$$\textsf{Ddif}(w)=\textsf{Dtop}(w)-\textsf{Dbot}(w).$$

Given a word $w=w_{1}w_{2}\ldots w_{n}$,
we separate $w$ into its \emph{descent blocks} by putting in dashes between $w_{i}$ and $w_{i+1}$ whenever $w_{i}\leqslant w_{i+1}$.
A maximal  contiguous subword of $w$ which lies between two dashes is a \emph{descent block}.
A descent block is an \emph{outsider} if it has only one letter;
otherwise, it is a \emph{proper} descent block.
The leftmost letter of a proper descent block is its \emph{closer} and the rightmost letter is its \emph{opener}.
Let $B$ be a proper descent block of the word $w$ and let $\text{C}(B)$ and $\text{O}(B)$
be the closer and opener of $B$, respectively.
If $a$ is a letter of $w$,
we say that $a$ is \emph{embraced by} $B$ if $\text{C}(B)\geqslant a> \text{O}(B)$.

The \emph{embracing numbers} of a word $w=w_{1}w_{2}\ldots w_{n}$ are the numbers $e_{1},e_{2},\ldots,e_{n}$ where $e_{i}$ is the number of descent blocks in $w$ that are strictly to the right of $w_{i}$ and that embrace $w_{i}$.
The \emph{embracing sum} of $w$, denoted by $\textsf{Res}(w)$, is defined by
$$\textsf{Res}(w)=e_{1}+e_{2}+\cdots+ e_{n}.$$
Define
\begin{align*}
\textsf{mak}(w)&=\textsf{Dbot}(w)+\textsf{Res}(w),\\
\textsf{mad}(w)&=\textsf{Ddif}(w)+\textsf{Res}(w).
\end{align*}
	\vskip -2mm

In \cite{Clarke-1997}, Clarke, Steingr\'{\i}msson, and Zeng first defined a bijection $\Phi$ on permutations and then extended it to words.
We begin with an overview of the bijection $\Phi$ on permutations.

Given a permutation $\pi\in\mathfrak{S}_{n}$,
we first construct two biwords,
$(\begin{smallmatrix} f\\ f^{\prime} \end{smallmatrix})$ and $(\begin{smallmatrix} g\\ g^{\prime} \end{smallmatrix})$,
and then form the  biword
$
\left(
 \begin{smallmatrix}
   f&g\\
   f^{\prime}&g^{\prime}
  \end{smallmatrix}
  \right)
$
by concatenating $f$ and $g$, and $f^{\prime}$ and $g^{\prime}$, respectively.
The word $f$ is defined as the subword of descent bottoms in $\pi$, ordered increasingly,
and the word $g$ is defined as the subword of non-descent bottoms in $\pi$, also ordered increasingly.
The word $f^{\prime}$ is the subword of descent tops in $\pi$,
ordered so that for any letter $x$ in $f^{\prime}$,
there are exactly $d$ letters in $f^{\prime}$ that are on the left of $x$ and that are greater than $x$,
where $d$ is the embracing number of the letter $x$ in $\pi$.
The word $g^{\prime}$ is the subword of non-descent tops in $\pi$,
ordered so that for any letter $x$ in $g^{\prime}$,
there are exactly $d$ letters in $g^{\prime}$ that are on the right of $x$ and that are smaller than $x$,
where $d$ is the embracing number of the letter $x$ in $\pi$.
Rearranging the columns of 
$
 \left(
 \begin{smallmatrix}
   f&g\\
   f^{\prime}&g^{\prime}
  \end{smallmatrix}
  \right)
$,
so that the top row is in increasing order,
then let $\Phi(\pi)$ be the bottom row  of the rearranged biword.

By the procedure of the construction of $\Phi$, we have the following result.
\begin{lemma}\label{lemma-CSZ}
The letters appearing in $f$ and $f'$ are precisely the excedance bottoms and excedance tops of $\Phi(\pi)$, respectively.	
\end{lemma}
\begin{proof}
This follows directly from the construction of $\Phi$; see the proof of Proposition 3 in \cite{Clarke-1997}.
\end{proof}

For a fixed composition $\alpha$, 
we extend $\Phi$ to a map on $\mathfrak{S}_{\alpha}$.
For $w\in\mathfrak{S}_{\alpha}$,
recall that $\text{std}(w)$ is the permutation in $\mathfrak{S}_{n}$  obtained by replacing, 
in the order of their appearance from left to right,
the $\alpha_{1}$ occurrences of the letter $1$ by $1,2,\ldots, \alpha_{1}$,
the $\alpha_{2}$ occurrences of the letter $2$ by $\alpha_{1}+1,\alpha_{1}+2,\ldots, \alpha_{1}+\alpha_{2}$, 
and similarly for the remaining letters.
Recall also that 
for a permutation $\pi\in\mathfrak{S}_{n}$,
$\text{istd}_{\alpha}(\pi)$ is the word in $\mathfrak{S}_{\alpha}$ obtained from $\pi$ by replacing each of $1,2,\ldots,\alpha_{1}$ with $1$,
each of $\alpha_{1}+1,\alpha_{1}+2,\ldots,\alpha_{1}+\alpha_{2}$ with $2$, and so on.
For $w\in\mathfrak{S}_{\alpha}$, we define  
$$\Phi_{\alpha}(w)=\text{istd}_{\alpha}\circ\Phi\circ\text{std}(w).$$

\begin{example}
\emph{Let $\alpha=(2,3,3)$. Consider the word
$$w=1~3~2~1~3~2~2~3\in\mathfrak{S}_{\alpha}.$$
Then
$$\pi=\operatorname{std}(w)=1-6~3~2-7~4-5-8.$$
One checks that
$$
 \left(
 \begin{matrix}
   f\\
   f^{\prime}
  \end{matrix}
  \right)=
  \left(
 \begin{matrix}
   2&3&4\\
   3&7&6
  \end{matrix}
  \right),~~~~
 \left(
 \begin{matrix}
   g\\
   g^{\prime}
  \end{matrix}
  \right)=
  \left(
 \begin{matrix}
   1&5&6&7&8\\
   1&2&4&5&8
  \end{matrix}
  \right).
$$
Hence,
$$
 \left(
 \begin{matrix}
   f&g\\
   f^{\prime}&g^{\prime}
  \end{matrix}
  \right)=
  \left(
 \begin{matrix}
   2&3&4&1&5&6&7&8\\
   3&7&6&1&2&4&5&8
  \end{matrix}
  \right)\longrightarrow
  \left(
 \begin{matrix}
   1&2&3&4&5&6&7&8\\
   1&3&7&6&2&4&5&8
  \end{matrix}
  \right).$$
Therefore,
  $$\Phi(\pi)=1~3~7~6~2~4~5~8.$$
Thus,
$$\Phi_{\alpha}(w)=\operatorname{istd}_{\alpha}(\Phi(\pi))=1~2~3~3~1~2~2~3.$$}
\end{example}

\begin{theorem}[Clarke--Steingr\'{\i}msson--Zeng \cite{Clarke-1997}]\label{Clarke-1997}
$\Phi_{\alpha}:\mathfrak{S}_{\alpha}\rightarrow\mathfrak{S}_{\alpha}$ is a bijection
satisfying
$$(\emph{\textsf{des}},\emph{\textsf{mak}},\emph{\textsf{mad}})w=
(\emph{\textsf{exc}},\emph{\textsf{den}},\emph{\textsf{inv}})
\Phi_{\alpha}(w)$$ for all $w\in\mathfrak{S}_{\alpha}$.
\end{theorem}
To prove that the bijection $\Phi_{\alpha}$ preserves the multiset of weak right-to-left minima, we give three lemmas.
\begin{lemma}\label{lemma-CSZ-Rlwmin-subset-perm}
	For any $\pi\in\mathfrak{S}_{n}$,
	we have $\emph{\textsf{Rlmin}}(\pi)\subseteq\emph{\textsf{Rlmin}}(\Phi(\pi))$.
\end{lemma}
\begin{proof}
Assume that $\pi=\pi_{1}\pi_{2}\ldots\pi_{n}$
and $\Phi(\pi)=\pi^{\prime}=\pi^{\prime}_{1}\pi^{\prime}_{2}\ldots\pi^{\prime}_{n}$.
Given $c\in\textsf{Rlmin}(\pi)$, we will prove that $c\in\textsf{Rlmin}(\pi^{\prime})$.
Since $c\in\textsf{Rlmin}(\pi)$, 
it follows that in permutation $\pi$, 
$c$ is a non-descent top and the embracing number of the letter $c$ is $0$.
By the definition of $\Phi(\pi)$,
we see that 

(i) $c\in g^{\prime}$, and 

(ii) in $g^{\prime}$, there is no letter that is on the right of $c$ and that is smaller than $c$.

\noindent Below we prove $c\in\textsf{Rlmin}(\pi^{\prime})$ by showing that for any $b$ with $b<c$, 
the letter $b$ lies to the left of $c$ in $\pi^{\prime}$. 

$\bullet$~$b\in g^{\prime}$. 
By (i), (ii) and  the fact that $b<c$, 
we have that $b$ lies to the left of $c$ in $g^{\prime}$.
Since $g$ is increasing, the relative order of the letters in $g^{\prime}$ is preserved in $\pi^{\prime}$.
Hence $b$ lies to the left of $c$ in $\pi^{\prime}$.

$\bullet$~$b\in f^{\prime}$. 
By Lemma \ref{lemma-CSZ}, the letters of $f^{\prime}$ are precisely the excedance tops of $\pi^{\prime}$, 
and hence the letters of $g^{\prime}$ are precisely the non-excedance tops of $\pi^{\prime}$.
Assume that $\pi_{j}^{\prime}=b$ and $\pi_{k}^{\prime}=c$.
Since $b\in f^{\prime}$, 
the letter $b$ is an excedance top of $\pi^{\prime}$,
and hence $j<b$.
Since $c\in g^{\prime}$, 
the letter $c$ is a non-excedance top of $\pi^{\prime}$,
and hence $c\leqslant k$.
Therefore, $j<b<c\leqslant k$.
Thus, $b$ lies to the left of $c$ in $\pi^{\prime}$.
\end{proof}
In the remainder of this section, when we write $w_p\in\textsf{Rlwmin}(w)$, 
we mean that the occurrence of the letter $w_p$ at position $p$ is a weak right-to-left minimum of $w$.
Similarly, when we write $\{w_{p_{1}},\ldots,w_{p_{k}}\}\subseteq\textsf{Rlwmin}(w)$,
we mean that each of the occurrences $w_{p_{1}},\ldots,w_{p_{k}}$ at the corresponding positions
is a weak right-to-left minimum of $w$.
\begin{lemma}\label{lemma-CSZ-std-istdM}
(1) Let $w=w_{1}w_{2}\ldots w_{n}\in\mathfrak{S}_{\alpha}$,
and let $\operatorname{std}(w)=\pi=\pi_{1}\pi_{2}\ldots\pi_{n}\in\mathfrak{S}_{n}$.
We have $w_{p}\in\emph{\textsf{Rlwmin}}(w)$ if and only if $\pi_{p}\in\emph{\textsf{Rlmin}}(\pi)$.

(2) Let $\pi^{\prime}=\pi^{\prime}_{1}\pi^{\prime}_{2}\ldots\pi^{\prime}_{n}\in\mathfrak{S}_{n}$,
and let $\operatorname{istd}_{\alpha}(\pi^{\prime})=u=u_{1}u_{2}\ldots u_{n}\in\mathfrak{S}_{\alpha}$.
If $\pi_{p}^{\prime}\in\emph{\textsf{Rlmin}}(\pi^{\prime})$, then $u_{p}\in\emph{\textsf{Rlwmin}}(u)$.
\end{lemma}
\begin{proof}
(1) By Lemma \ref{lemma-std}.

(2) By the definition of  $\text{istd}_{\alpha}$,
we see that if $\pi_{i}^{\prime}<\pi_{j}^{\prime}$, then $u_{i}\leqslant u_{j}$,
which yields the desired result.
\end{proof}

\begin{lemma}\label{lemma-CSZ-Rlwmin-subset}
For any $w\in\mathfrak{S}_{\alpha}$,
we have $\emph{\textsf{Rlwmin}}(w)\subseteq\emph{\textsf{Rlwmin}}(\Phi_{\alpha}(w))$as multisets.
\end{lemma}
\begin{proof}
    Assume that 
    \begin{align*}
    w=w_{1}w_{2}\ldots w_{n},\quad
    \text{std}(w)&=\pi=\pi_{1}\pi_{2}\ldots\pi_{n},\quad
    \Phi(\pi)=\pi^{\prime}=\pi^{\prime}_{1}\pi^{\prime}_{2}\ldots\pi^{\prime}_{n},
    \end{align*}
    and that
    $$\Phi_{\alpha}(w)=\text{istd}_{\alpha}(\pi^{\prime})=u=u_{1}u_{2}\ldots u_{n}.$$
	Let $\{s^{k}\}\subseteq\textsf{Rlwmin}(w)$. 
	We will prove that $\{s^{k}\}\subseteq\textsf{Rlwmin}(u)$.
	Let $w_{p_{1}},w_{p_{2}},\ldots,w_{p_{k}}$ be the rightmost $k$ occurrences of the letter $s$ in $w$, where $p_{1}<p_{2}<\cdots<p_{k}$.
	Since ${\{s^k\}\subseteq \mathsf{Rlwmin}(w)}$ and any occurrence of $s$ to the right of a weak right-to-left minimum equal to $s$ is again a weak right-to-left minimum, we have
    $$\{w_{p_{1}},w_{p_{2}},\ldots,w_{p_{k}}\}\subseteq\textsf{Rlwmin}(w).$$ 
    By Lemma \ref{lemma-CSZ-std-istdM} (1), we have $$\{\pi_{p_{1}},\pi_{p_{2}},\ldots,\pi_{p_{k}}\}\subseteq\textsf{Rlmin}(\pi).$$
    It follows from Lemma \ref{lemma-CSZ-Rlwmin-subset-perm} that $$\{\pi_{p_{1}},\pi_{p_{2}},\ldots,\pi_{p_{k}}\}\subseteq\textsf{Rlmin}(\pi^{\prime}).$$
    Let $q_{i}$ be such that $\pi^{\prime}_{q_{i}}=\pi_{p_{i}}$ for all $1\leqslant i\leqslant k$.
    By Lemma \ref{lemma-CSZ-std-istdM} (2),
    we have $$\{u_{q_{1}},u_{q_{2}},\ldots,u_{q_{k}}\}\subseteq\textsf{Rlwmin}(u).$$
Since $w_{p_i}=s$ and $\pi=\mathrm{std}(w)$, the letters $\pi_{p_i}$ correspond to the occurrences of the letter $s$ in $w$. Hence
\[
u_{q_i}=\mathrm{istd}_\alpha(\pi'_{q_i})
=\mathrm{istd}_\alpha(\pi_{p_i})
=s
\]
for all  $1\leqslant i\leqslant k$.
Therefore, $\{s^{k}\}\subseteq\textsf{Rlwmin}(u)$.
\end{proof}
\begin{proposition}\label{Prop-CSZ-keeps-Rlwmin}
For any $w\in\mathfrak{S}_{\alpha}$,
we have $\emph{\textsf{Rlwmin}}(w)=\emph{\textsf{Rlwmin}}(\Phi_{\alpha}(w))$.
\end{proposition}
\begin{proof}
Since $\Phi_{\alpha}$ is a bijection on the finite set $\mathfrak{S}_{\alpha}$,
for any $w\in\mathfrak{S}_{\alpha}$, there exists a positive integer $k$ such that
$$w\xrightarrow{\Phi_{\alpha}}\Phi_{\alpha}(w)\xrightarrow{\Phi_{\alpha}}\Phi_{\alpha}^{2}(w)\xrightarrow{\Phi_{\alpha}}\cdots\xrightarrow{\Phi_{\alpha}}\Phi_{\alpha}^{k}(w)=w.$$
By Lemma \ref{lemma-CSZ-Rlwmin-subset},	we have
$$\textsf{Rlwmin}(w)\subseteq\textsf{Rlwmin}(\Phi_{\alpha}(w))\subseteq\textsf{Rlwmin}(\Phi_{\alpha}^{2}(w))\subseteq\cdots\subseteq\textsf{Rlwmin}(\Phi_{\alpha}^{k}(w))=\textsf{Rlwmin}(w).$$
Hence all these multisets are equal.
This completes the proof.
\end{proof}

By Theorem \ref{Clarke-1997} and Proposition \ref{Prop-CSZ-keeps-Rlwmin},
the triples $(\textsf{des},\textsf{mak},\textsf{mad})$  and 
$(\textsf{exc},\textsf{den},\textsf{inv})$ are equidistributed on $\mathfrak{S}_{\alpha,R}$.

\section*{Acknowledgment}
\addcontentsline{toc}{section}{Acknowledgment} 
This work was supported by the National Natural Science Foundation of China (12101134).

\phantomsection
\begin{spacing}{0.5}
	
\end{spacing}
\end{document}